\documentclass[10pt,reqno]{amsart} 

\usepackage{amssymb,latexsym}
\usepackage{cite} 

\usepackage[height=190mm,width=130mm]{geometry} 

\theoremstyle{definition}
\theoremstyle{plain}
\usepackage{amsfonts}
\usepackage{amsmath}
\usepackage{amscd}
\usepackage{amsthm}
\usepackage{bbm}
\usepackage{bm}
\usepackage{enumerate}
\usepackage{enumitem}
\usepackage{hyperref}
\usepackage{mathabx}
\usepackage{mathdots}
\usepackage{mathtools}

\usepackage{citehack}
\usepackage{longtable}
\usepackage[matrix,arrow,curve]{xy}

\usepackage{lipsum}

\allowdisplaybreaks

\newcommand{\C}{\mathbb{C}}

\newcommand{\F}{\mathcal{F}}

\newcommand{\wt}{\textup{wt}}

\theoremstyle{plain}

\newtheorem{lemma}{Lemma}

\theoremstyle{definition}

\theoremstyle{remark}

\newcommand{\Z}{\mathbb{Z}}
\newcommand{\V}{\mathcal V}

\numberwithin{equation}{section} 

\newcommand{\rmap}{\longrightarrow}

\newcommand{\U}{\ensuremath{\mathcal{U}}}

\newcommand{\g}{\ensuremath{\Gamma}}

\newcommand{\ps}{{\raise 1pt\hbox{\tiny (}}}

\newcommand{\pss}{{\raise 1pt\hbox{\tiny [}}}
\newcommand{\pdd}{{\raise 1pt\hbox{\tiny ]}}}
\newcommand{\pd}{{\raise 1pt\hbox{\tiny )}}}

\newcommand{\bs}{{\raise 1pt\hbox{\tiny [}}}
\newcommand{\bd}{{\raise 1pt\hbox{\tiny ]}}}

\def\cross{\mathinner{\mathrel{\raise0.8pt\hbox{$\scriptstyle>$}}
                 \joinrel\mathrel\triangleleft}}

\def\U{\mathcal{U}}
\def\W{\mathcal{V}}
\def\W{\mathcal{W}}
\def\K{\mathcal{K}}

\usepackage{stackrel}

\newcommand{\be}{\begin{equation}}
\newcommand{\ee}{\end{equation}}

\newcommand{\nn}{\nonumber \\}

 \newcommand{\res}{\mbox{\rm Res}}

\newcommand{\one}{\mathbf{1}}

\newcommand{\nc}{\newcommand}
\nc{\cali}{\mathcal}
\nc{\on}{\operatorname}
\nc{\Wick}{{\mb :}}

\nc{\ddz}{\frac{\partial}{\partial z}}
\nc{\ch}{\mbox{ch}}
\nc{\Oo}{{\cali O}}
\nc{\cond}{|\,}
\nc{\bib}{\bibitem}
\nc{\pone}{\Pro^1}
\nc{\pa}{\partial}
\nc{\arr}{\rightarrow}
\nc{\larr}{\longrightarrow}
\nc{\ket}{)}
\nc{\bra}{(}
\nc{\gam}{\bar{\gamma}}
\nc{\ep}{\epsilon}
\nc{\su}{\widehat{{\mf s}{\mf l}}_2}
\nc{\sw}{{\mf s}{\mf l}}
\nc{\h}{{\mf h}}
\nc{\n}{{\mf n}}
\nc{\ab}{\mf{a}}
\nc{\is}{{\mb i}}
\nc{\js}{{\mb j}}
\nc{\He}{{\cali H}}
\nc{\inv}{^{-1}}
\nc{\ol}{\overline}
\nc{\wh}{\widehat}
\nc{\dst}{\displaystyle}

\nc{\delt}{\partial_t}
\nc{\ddt}{\frac{\partial}{\partial t}}
\nc{\delx}{\partial_x}
\nc{\mb}{\mathbf}
\nc{\mf}{\mathfrak}

\nc{\mbb}{\mathbb}
\nc{\Ctt}{\C((t))}
\nc{\Ct}{\C[t,t\inv]}

\nc{\ghat}{\wh{\g}}

\nc{\un}{\underline}
\nc{\mc}{\mathcal}
\nc{\BB}{{\mc B}}
\nc{\bb}{{\mf b}}
\nc{\kk}{{\mf k}}
\nc{\frob}{\times}
\nc{\sm}{\setminus}
\nc{\Pp}{{\mathbb P}^1}
\nc{\Aa}{{\mc A}}

\nc{\AutO}{\on{Aut}\Oo}
\nc{\AUTO}{\un{\on{Aut}}\Oo}
\nc{\AUTK}{\un{\on{Aut}}\K}
\nc{\Heout}{\He_{\out}}
\nc{\Hetil}{{\widetilde\He}}
\nc{\wb}{\overline}

\nc{\Res}{\on{Res}}
\nc{\pitil}{\Pi}
\nc{\Ctil}{\wt{C}}
\nc{\auto}{\on{Aut} \Oo}
\nc{\phitil}{\wt{\phi}}
\nc{\gz}{\g_{\vec z}}
\nc{\tensorM}{\bigotimes_{i=1}^N{\mathbb M}_i}
\nc{\tensorW}{\bigotimes_{i=1}^N W_{\nu_i,k}}
\nc{\out}{\on{out}}

\nc{\m}{{\mathfrak m}}

\nc{\gx}{\g^0_{\vec x}}

\nc{\hx}{\He^0_{\vec x}}
\nc{\tensorpi}{\pi_{\nu_1,\ldots,\nu_N}^\kappa}
\nc{\Phizw}{\Phi_{\vec w}({\vec z})}
\nc{\Pro}{{\mathbb P}}

\nc{\De}{D}

\nc{\us}{\underset}

\nc{\Ll}{\mc L}
\nc{\dR}{\on{dR}}

\nc{\T}{{\mc T}}

\nc{\Xn}{\overset{\circ}X{}^n} \nc{\Dn}{\overset{\circ}D{}^n}
\nc{\Dxn}{\overset{\circ}D{}^n_x} \nc{\varphitil}{\wt{\varphi}}

\nc{\lf}{{\mf l}}
\nc{\Wir}{\on{Vir}}

\nc{\bfgn}{{\bf g}_n}
\nc{\bfzn}{{\bf z}_n}

\begin{document}
\title[On holonomy groupoid of vertex operator algebra bundles]
{On holonomy groupoid of vertex operator algebra bundles on foliations}
                                
\author{A. Zuevsky} 
\address{Institute of Mathematics \\ Czech Academy of Sciences\\ Zitna 25, Prague\\ Czech Republic}

\email{zuevsky@yahoo.com}

\begin{abstract}
For a foliation $\F$ defined on a smooth complex manifold $M$ 
 we introduce the category of vertex operator algebra $V$ 
bundles with sections 
provided by vectors of elements of the space of 
algebraically extended $V$-module $W$-valued differentials. 
An intrinsic coordinate-independent formulation for such bundles  
 is given. 
Finally, we identify the cohomology of the spaces of sections for a vertex operator algebra $V$ 
bundle with vertex operator algebra cohomology of the 
holonomy groupoid $Hol(M, \F)$. 

\bigskip 
AMS Classification: 53C12, 57R20, 17B69 
\end{abstract}

\keywords{Holonomy groupoids, fiber bundles, vertex operator algebras, cohomology}

\vskip12pt  

\maketitle

\section{Conflict of Interest}
The author states that: 

1.) The paper does not contain any potential conflicts of interests.

\section{Data availability statement}
The author confirms  that: 

\medskip 
1.) All data generated or analysed during this study are included in this published article. 

\medskip 
2.)   Data sharing not applicable to this article as no datasets were generated or analysed during the current study.

\section{Introduction: results of the paper}  
\label{introduction}
The theory of foliated manifolds incorporates a few main approaches 
\cite{Bott, BS, BH, CM, Fuks, F73, LosikArxiv}.     
The idea of studies of foliations cohomology, cohomology of related bundles, and connections to 
 to cohomology of foliated manifolds themselves was proposed in \cite{BS}.  
  Let $Vect(M)$ be the Lie algebra of vector fields on $M$.   
In \cite{BS} is was proven 
  that the Gelfand-Fuks cohomology $H^*(Vect(M))$ \cite{Fuks}  
 is isomorphic to the singular cohomology $H^*(E)$ 
of the space $E$ of continuous cross sections of a certain 
fiber bundle $\mathcal E$ over $M$. 
In \cite{PT, Sm} they continued to use advanced topological methods of \cite{BS} 
 for cases of more general cosimplicial spaces of maps.   
In \cite{Wag} it was demonstrated that the ordinary theory 
 of vector fields on a complex manifold $M$ was not always the most effective way 
to study cohomology of $M$. 
One has to $M$ consider more complicated algebraic and geometrical structures to 
arrive at non-trivial cohomology theories associated to such structures. 
One of possible candidate for such structures is given by vertex operator algebras 
with formal parameters considered as local coordinates on complex manifolds. 
 Vertex operator algebras \cite{B, DL, K} represent generalizations of ordinary Lie algebras and 
constitute an essential part of conformal field theory \cite{FMS, BZF, H2}.   

The main motivation for studies of this paper
 is to develop a vertex operator algebra approach 
to cohomology of auxiliary bundle defined on leaf spaces $M/\F$ and transversal sections of foliations. 
The ground idea it to use well-developed and powerful machinery and  
structural and computation properties of vertex operator algebras 
to cohomology of non-commutative objects attached to $M/\F$ 
to describe its leaves in terms of corresponding invariants.  
By taking into account the standard methods of defining canonical
 cosimplicial object \cite{Fei, Wag} as well as the ${\rm \check{C}}$ech-de Rham 
cohomology construction \cite{CM},  
 we construct in this paper in intrinsic coordinate-independent way
 canonical fiber bundles associated to $\F$ on $M$. 
Our main purpose then is to demonstrate in Lemma \ref{groupo} 
that there exist a computational vertex operator algebra based way to 
determine the cohomology of the holonomy groupoid $Hol(M, \F)$ for a foliation 
$\F$ on $M$. 
Developing \cite{BS}, the cohomology of foliations   
for a complex smooth manifold $M$ is expressed in terms of cohomology 
of a canonical complex  
for an auxiliary bundle  
 with intrinsic action of the coboundary operator.  
The construction of bundles with canonical sections defined over abstract discs 
 on a smooth complex manifold $M$ is
 grounded on the structure of admissible vertex operator algebra $V$ modules $W$. 
Corresponding cohomology is considerd in terms of spaces of rational functions 
provided by values of non-degenerate bilinear pairings on $W$ with 
 specified analytic behavior,  
 and satisfying certain symmetry properties. 

The content and main results of this article are as follows. 
In order to give a local description of leaves of a foliation $\F$ of an $n$-dimensional smooth manifold $M$
we use the classical approach of transversal sections as well as algebraic and analytic properties of 
vertex operator algebras. 
We chose two sets of points on $M$ and on a basis $\U$ of transversal sections with corresponding domains of 
local coordinates. 
Points on $M$ and $\U$ are then endowed with sets of a $V$ vertex operator algebra elements.  
By taking algebraic completions $\overline{W}$ of elements $W$ of the category ${\mathcal W}$ of $V$-modules, 
we formulate the definition of spaces $\widehat{W}^q_r$, $q$, $r \ge 0$,  
of special vectors ${\bf X}$ (exlicitely defined in \eqref{overphi0}) of $\overline{W}$-valued 
rational forms combined with sets of vertex operators.    
For a set of formal complex variables $(z_1, \ldots, z_s)$ 
we introduce the space $\widehat{W}_{(z_1, \ldots, z_s)}$ of algebraic completion of 
the graded (with respect to Virasoro algebra $L_W(0)$-mode) 
  space of differential form-valued vectors 
\begin{equation}
\label{overphi0}
{\bf X}(v_1, z_1; \ldots;  v_s, z_s)= \left[  
X\left(v_1, z_1 dz_{ {\it i}(1)}; 
\ldots; v_s, z_s \; dz_{{\it i}(s)} \right) \right], 
\end{equation}       
where ${\it i}(j)$, $j=1, \ldots, s$, are cycling permutations of $(1, \dots, s)$ starting with $j$, 
and we denote by $\left[.\right]$ the vector with elements given by mappings $X$.  
In cases where it is clear which set of formal variables is used we skip $(z_1, \ldots, z_s)$ from 
notations and denote $\widehat{W}_{(z_1, \ldots, z_s)}$ as $\widehat{W}$.  
Assuming that there exists a non-degenerate bilinear pairing $(. , .)$ on $\widehat{W}_{(z_1, \ldots, z_s)}$,   
we denote by 
$\widehat{W}^*_{(z_1, \ldots, z_s)}$ the dual to $\widehat{W}_{(z_1, \ldots, z_s)}$  with respect to $(. , .)$.  
In case when elements $(z_1, \ldots,  z_s)$ are associated to certain local coordinates of 
 $l$ points $(p_1, \ldots, p_l)$ on $M$, 
we denote $\widehat{W}_{(z_1, \ldots, z_s)}$ by $\widehat{W}_{(p_1, \ldots, p_l)}$, 
and when $(z_1, \ldots,  z_s)$
are substituted by local coordinates $(t_{p_1}, \ldots, t_{p_s})$ in vicinities  
of $(p_1, \ldots, p_l)$, we replace $\widehat{W}_{(z_1, \ldots, z_s)}$ by    
 by $\widehat{W}_{(t_{p_1}, \ldots,t_{p_s})}$.  
For fixed $\theta \in \widehat{W}^*_{(z_1, \ldots, z_s)}$,   
and varying elements of $\widehat{W}_{(z_1, \ldots, z_s)}$ we  
 consider a vector of matrix elements of the form 
\begin{equation}
\label{toppa}
\Omega(X(v_1, z_1; \ldots; v_s, z_s))
 = \left( \theta, X(v_1, z_1; \ldots; v_s, z_s) \right)  \in \C((z)),  
\end{equation}
where $X(v_1, z_1; \ldots; v_s, z_s)$ r
depends implicitly on $v_i \in V$, $1 \le i \le s$.  
We may view the vector ${\bf X}(v_1, z_1; \ldots; v_s, z_s)$ of the space $\widehat{W}$ as a 
section of a fiber bundle over a collection of non-intersecting punctured discs 
 $(D^\times_{z_1}, \ldots, D^\times_{z_s} )= 
(Spec_{z_j}\; \C( (z_j) )$, $1 \le j \le s$,  
 with an ${\rm End} \left( \widehat{W}_{(z_1, \ldots, z_s)}\right)$-valued fiber 
$X(v_1, z_1; \ldots; v_s, z_s) \in \widehat{W}_{(z_1, \ldots, z_s)}$.  
In this paper we explain how to construct the vertex operator algebra $V$-bundle mentioned above 
 in the case when it 
 carries an action of the group ${\rm Aut}_l \;  \Oo^{(n)}$ of local  
 coordinates changes in vicinities of $l$ points on $M$.   
 This means that the action of the group 
${\rm Aut}_l\; \Oo^{(n)}   
={\rm Aut}_1\; {\Oo}^{(n)} \times \ldots \times {\rm Aut}_l\; {\Oo}^{(n)}$  
comes about by exponentiation 
of the action of vertex operator algebra $\left({\it Der}_j \; \Oo^{(n)} \right)$,  
 $1 \le j \le l$,   
via the action on $\widehat{W}_{(z_1, \ldots, z_s)}$.   
 The representation in term of formal series in $(t_{ p_1}, \ldots, t_{ p_l})$ 
   allows us to find the precise transformation formula 
 for all elements of $\widehat{W}_{(p_1, \ldots, p_l)}$
 under the action of ${\rm Aut}_l\; \Oo^{(n)}$.   
We then use this formula to give an intrinsic geometric meaning to sections 
${\bf X}(p_1, \ldots, p_s)$ of the fiber bundle in coordinate-free formulation.   
 Namely, we
attach to each admissible vertex operator algebra module $V$-module $W$ (i.e., 
satisfying certain properties)    
 a fiber bundle $\W_M$ on an arbitrary smooth manifold $M$.  
In Section \ref{bundle} we show that the bundle $\W_{M/\F}$ constructed is canonical, i.e., its 
sections do not depend on changes 
$(t_{ p_1}, \ldots, t_{ p_l}) \mapsto ( \widetilde{t}_{ p_1}, \ldots, \widetilde{t}_{ p_l} )$ of coordinates   
around points $(p_1, \ldots, p_l)$ on $M$.  
 To keep elements of $\widehat{W}^q_r$ coherent with respect to actions of the coboundary operators $\Delta^q_r$
 shifting indices $q$ and $r$, 
we apply certain analytic restriction on their characteristics provided by values of non-degenerate bilinear forms 
of entries in $\widehat{W}^q_r$-vectors.  
The spaces $\widehat{W}^q_r$ are defined on cosimplicial domains chosen on transversal sections of $\F$. 
Then we formulate definition of   
the category of vertex operator algebra bundles defined on $M/\F$.  
 The spaces $\widehat{W}^q_r$ associated to the category ${\mathcal W}$ of admissible $V$-modules
defined in Section \ref{cosimplicial}.    
The spaces $C^q_r$ of vectors of characteristics $[\Omega{X}]$ of entries of vectors ${\bf X}$  
form \cite{Huang} a double chain-cochain complex $(C^q_r, \delta^q_r)$ 
where $\delta^q_r {\bf X} =[\Omega( \Delta^q_r X)]$.  
The standard definition of cohomology of this complex is taken as cohomology of $Hol(M, \F)$. 
We show that elements of the spaces $\widehat{W}$ are invariant torsors with 
respect to the group of foliation preserving changes of transversal basis and local coordinates. 
Though the construction of $\W_{M/\F}$-bundle does not depend neither on the choice of transversal basis 
nor on the choice of coordinates on $M$,  
it does depend on the choice of vertex operator algebra elements  
 as well as on a particular element of the category
 $\W$ of admissible $V$-modules.        
The construction involves 
torsors and twists of a vertex operator algebra modules by the group of 
automorphisms of local coordinates transformations (independent for each chosen point 
on leaves of a foliation $\F$) of non-intersecting 
domains 
of a number of points on $M$.  

The plan of the paper is the following. 
Section \ref{puzo} contains information on vertex operator algebras, their modules and properties. 
In Section \ref{toss} we recall, following \cite{BZF}, the standard definitions of 
differentials and rational functions considered on abstract and standard discs.  
In Section \ref{bundle} we consider the general notion of a vertex operator algebra bundle $\W_{M/\F}$ 
defined on the leaf space of a foliated smooth complex manifold $M$.   
In Section \ref{cosimplicial} the category of vertex operator algebra bundles on leafs of $M/\F$ 
and transversal sections for a foliation $\F$ on $M$ 
is considered. 

There exists a bunch of ways to apply the construction of this paper of a vertex operator algebra bundle 
on the space of leaves for a foliation defined on a smooth manifold. 
The first obvious aim is to apply this study to 
  techniques of codimension one foliations discussions reflected in 
  \cite{Ghys, BG, BGG, Gal}. 
The problem of finding non-vanishing cohomological invariants for the space of a foliation leaves,
and the problem of distinguishing kinds of compact and non-compact leaves 
examples of foliations 
(such as the Reeb foliation of the full torus), 
are among important questions in the theory of foliations. 
The category ${\mathfrak W}_{M/\F}$ introduced in this paper 
for vertex operator algebra bundles $\W_{M/\F}$   
defined on leaf spaces of foliations  
will be used in establishing corresponding characteristic classes theory.   
One would be interested in finding possible relations of the cohomology theory of this paper 
with the chiral de-Rham complex on a smooth manifold introduced in \cite{MSV}.  
 We are also able to provide  
 applications of vertex operator algebra $V$-bundles $\W_{M/\F}$ for foliations 
of complex manifolds  
\cite{Bott, Fei, Wag,  FMS, TUY} 
in deformation theory \cite{Ma, BG, HinSch, GerSch, Kod},  
and algebraic topology in general.   
Constructions introduced in this paper will be useful for purposes 
of cosimplitial cohomology \cite{Wag} of manifolds. 
Vertex operator algebra bundles on complex manifolds   
can be used in construction of various generalizations of the Bott--Segal theorem \cite{BS}.  
Finally, we would like to mention possible connections to the Losik's theory of 
foliated manifolds. 
In \cite{Lo} Losik has introduced a smooth structure on the leaf space 
$M/\F$ of a foliation $\F$ of codimension $p$ on
a smooth manifold $M$ that allows to apply to $M/\F$ the
same techniques as to smooth manifolds.
Characteristic classes for 
foliations as elements of the cohomology of certain bundles over the
leaf space $M/\F$ were defined.   
We hope to develop this approach by applying vertex operator algebra techniques and constructions 
provided in this paper. 
\section{Vertex operator algebras and their modules}  
\label{puzo} 
In this Section we recall definitions and basic properties of 
 vertex operator algebras and their generalized modules 
 \cite{B, BZF, DL, FLM, BZF, H2,  K}.  
A vertex operator algebra   
$(V,Y_V,\mathbf{1}_V, c)$, of Virasoro algebra central charge $c$, 
  consists of a $\Z$-graded complex vector space 
$V = \bigoplus_{s\in\Z} V_{(s)}$, with finite-dimensional grading subspaces $V_{(s)}$ 
$\dim V_{(s)}<\infty$ for each $s\in \Z$,   
equipped with 
a linear map  
$Y_V: V\rightarrow {\rm End }(V)[[z,z^{-1}]]$,  
 for a formal complex parameter $z$ and a 
distinguished vector $\mathbf{1}_V \in V$.  
The vertex operator for $v\in V$ is given by 
$Y_V(v,z) = \sum_{s\in\Z}v(s)z^{-s-1}$,    
with components $(Y_V(v))_s=v(s)\in {\rm End \;}(V)$,
 with the property $Y_V(v,z)\mathbf{1}_V = v+O(z)$.
In this paper we apply the following restrictions on the grading 
of a vertex operator algebra $V$ or 
its module $W$. 
A vertex operator algebra $V$-module $W$ 
 is a vector space 
$W$ equipped with a vertex operator map  
$Y_W: V\otimes W \to W[[z, z^{-1}]]$, and  
 $v \otimes w \mapsto  Y_W(v, z)w=\sum_{s\in \Z}(Y_W)_s(v,w)z^{-s-1}$. 
 $W$ is also subject of actions of 
and linear operators $L_W(0)$ and $L_W(-1)$ ($0$ and $-1$ Virasoro modes) 
  satisfying the following conditions. 
One assumes that $V_{(s)}=0$ for $s\ll 0$.  
The vector space $W$ is $\mathbb C$-graded, that is, 
$W=\bigoplus_{\alpha \in \C } W_{(\alpha)}$,  
such that $W_{(\alpha)}=0$, when the real part of $\alpha$ is sufficiently negative.  
The result of a vertex operator $Y_{V, W}(u, z)(v, w)$, 
$u$, $v\in V$, $w\in W$,  contains only finitely many negative power terms,
that is 
 $Y_{V, W}(u, z)(v, w) \in (V, W)((z))$, i.e., belongs to the space of formal 
Laurent series in $z$ with coefficients in $(V, W)$. 
 Here $(V, W)$ and subscript ${}_{V, W}$ 
mean corresponding expression either for vertex operator algebra $V$ elements
or its module $W$.  
Let ${\rm Id}_{V, W}$ be the identity operator on $(V, W)$.
Then $Y_{V, W}(\mathbf{1}_V, z)={\rm Id}_{V, W}$.   
For $v\in V$, $Y_V(v, z)\mathbf{1}_V\in V[[z]]$ 
and $\lim_{z\to 0}Y_V(v, z)\mathbf{1}_V=v$. 
We assume that for $W$ there exist non-degenerate bilinear pairing $(.,.)$, $W'\otimes W \to \C$, 
where $W'$ denotes the dual $V$-module to $W$.  
 For $s\in \Z_+$, denote by $F_s\mathbb C$ the 
configuration space of $s$ ordered points in $\mathbb C$, 
$F_s\mathbb C=\{(z_1, \ldots, z_s)\in \mathbb C^s,  z_i \ne z_j, i\ne j\}$. 
For a space $A$ and arbitrary $\theta \in A^*$, for $A^*$ dual to $A$, 
a meromorphic function of several complex variables $(z_1, \dots, z_s)$ 
defined by a map  
$f: F_s \mathbb C \rightarrow A$,    
$(z_1, \ldots, z_s) \mapsto  R( f(z_1, \ldots, z_s))$,    
is called an $A$-valued rational function
if its characteristic  
$\Omega(f(z_1, \ldots, z_s) ) = (\theta, f(z_1, \ldots, z_s))$,
extends to a rational function denoted $R( f(z_1, \ldots, z_s))$ in $(z_1, \ldots, z_r)$
   on a larger domain and   
 admits poles at   
$z_i=z_j$, $i\ne j$, only.  
In particular, in this paper we consider the cases $A=\overline{W}$, $\widehat{W}$.  
We assume that for $u$, $u_1$, $u_2 \in V$, the characteristics  
 $\Omega\left(Y_{V, W}(v_1, z_1) Y_{V, W}(v_2, z_2)v \right)$,  
 $\Omega \left(Y_{V, W}(v_2, z_2)Y_{V, W}(v_1, z_1)(v, w)\right)$,  and  
 $\Omega \left(Y_{V, W}(Y_V(v_1, z_1-z_2)v_2, z_2)(v, w)\right)$, 
 converge absolutely
in the regions $|z_1|>|z_2|>0$, $|z_2|>|z_1|>0$, $|z_2|>|z_1-z_2|>0$,   
correspondingly to a common rational function  
in $z_1$, $z_2$. 
Poles of these characteristics are only allowed at $z_1=0=z_2$, and $z_1=z_2$.   
The role of a grading operator for $V$ is played by the zero Virasoro mode with 
$L_V(0)v=rv$ for $v\in V_{(r)}$.
 Then for $v\in V$ one has 
\[
[L_{V, W}(0), Y_{V, W}(v, z)]=Y_{V, w}(L_V(0)v, z)+z\frac{d}{dz}Y_{V, W}(v, z). 
\]  
For $w\in W_{(\alpha)}$, there exists
$n_0 \in \Z_+$ such that 
$(L_W(0)-\alpha)^{n_0} w=0$.   
For $v\in V$ the operator $L_V(-1)$ is 
given by 
\[
L_V(-1)v=\res_z z^{-2}Y_{V}(v, z)\one_V=Y_{(-2)}(v){\bf 1}_V, 
\] 
\[
\frac{d}{dz}Y_V(v, z)=Y_{V, W}(L_V(-1)v, z)=[L_{V, W}(-1), Y_V(v, z)].
\]  
 We denote $\wt(v)=k$ the weight for $v\in V_{(k)}$. 
For $v\in V$, the translation property for vertex operators can be written as
\[
 Y_W(v, z) = e^{-z' L_W(-1)} Y_{W}(v, z+z') e^{z' L_W(-1)}, 
\]
 $z' \in \mathbb C$.
For $v\in V$, it follows  
\[
\frac{d}{dz}Y(v, z)=Y(L_V(-1)v, z).
\] 
For $a \in \C$, the conjugation property with respect to the grading operator $L_W(0)$ is given by 
\[
 a^{ L_W(0) } \; Y_W(v,z) \; a^{  -L_W(0) }= Y_W \left(a^{ L_W(0) } v, az\right).
\]     
A vertex operator algebra $V$ satisfying conditions above is called conformal of central  
charge $c \in \C$, 
if there exists a non-zero conformal vector $\omega \in V_2$ such that the
Fourier coefficients $L_V(r)$ of the corresponding vertex operator
$Y (\omega, z) = \sum_{s \in \Z} L_V(k) z^{-s-2}$,   
is determined by Virasoro modes $L_V(r): V\to V$ subject to the commutation relations 
\[
[L_V(s), L_V(r)]=(s-r)L_V(s+r)+\frac{c}{12}(s^3-s)\delta_{s, -r}\; {\rm Id}_V.
\]  
In \cite{BZF}, $v \in V$, the following formula was derived  
\[
\left[L_W(r), Y_W (v, z) \right] 
=  \sum_{r \geq -1}  
 \frac{1}{(r+1)!} \partial^{r+1}_z z^{r+1}  
\;  
Y_W (L_V(r) v, z). 
\]     
For a vector field 
$\beta(z)\partial_z= \sum_{r \geq -1} \beta_r z^{r+1} \partial_z$,    
 $\beta(z)\partial_z \in {\rm Der} \Oo^{(1)}$,  
which belongs to the local Lie algebra of the group ${\rm Aut}\; \Oo^{(n)}$,    
 let us introduce the operator 
${\overline \beta }= - \sum_{r \geq -1} \beta_r L_W(r)$. 
In \cite{BZF} they prove the following formula:     
\[
\left[{\overline \beta}, Y_W (v, z) \right]  
=  \sum_{r \geq -1}  
 \frac{1}{(r+1)!} \left(\partial^{r+1}_z \beta(z)  \right)\;  
Y_W (L_V(r) v, z).
\]    
A vertex operator algebra $V$-module $W$ is called quasi-conformal if 
  it carries an action of ${\rm Der}\; \Oo^{(n)}$ on an $n$-dimensional smooth manifold $M$ 
such that commutation formula above 
 holds for any 
$v\in  V$, and $z=z_j$, $1 \le j \le n$, the element $L_W(-1) = - \partial_z$     
acts as the translation operator 
$L_W(0) = - z \partial_z$,    
 acts semi-simply with integral
eigenvalues, and the Lie subalgebra ${\rm Der}_+ \; \Oo^{(n)}$ acts locally nilpotently on $M$. 
 A vector $w \in W$ of a quasi-conformal 
vertex operator algebra $V$ is called 
 primary of conformal dimension $\nu \in  \mathbb Z_+$ if  
$L_W(k) w = 0$, $k > 0$,  
 $L_W(0) \cdot w = \nu w$.   
In addition to that, 
we assume that $V$-module $W$ admits an action of $Der_j \;  \Oo^{(n)}$.     
The element $\left(- \partial_{t_p}\right)$ plays a role of the translation operator 
  on $W_{t_p}$ with integral eigenvalues, and 
  the Lie subalgebra $\left(Der_+\right)_j \; \Oo^{(n)}$ acts locally nilpotently.  
The $\C$-grading operator is provided by the mode $L_W(0)$, i.e., 
 $L_W(0)=\left(- t_p \; \partial_{t_p}\right)$.    
Finally, let us assume that  
the action of the Lie algebra $Der_j\; \Oo^{(n)}$ on $\widehat{W}_{  ( t_{ p_1}, \ldots, t_{ p_l}  )  }$  
 can be exponentiated to an action of the
group ${\rm Aut}_j\; \Oo^{(n)}$.    

Denote by ${\mathfrak \W}$ the category of $V$-vertex operator algebra admissible modules $W$ 
that satisfy these properties in addition to  
all related properties of Section \ref{puzo}. 
Denote by $\overline{W}$ the algebraic completion of $W$,  
$\overline{W}=\prod_{r\in \mathbb C} W_{(r)}=(W')^{*}$.  
\section{Differentials and rational functions on abstract discs}
\label{toss}
In this Section we partially follow \cite{BZF} and describe 
 the setup needed for formulation of further results.  
Let $p$ be a point on $M$, and $t_p$ be a local coordinate in a vicinity of $p$.   
 We replace the field of Laurent series $\C((t_p))$ 
by any complete topological algebra non-canonically isomorphic to 
$\C((t_p))$.
\subsection{Abstract discs}
To introduce abstract discs on $M/\F$ it is possible to 
  consider the scheme underlying the $\C$-algebra $\C[[t_p]]$.   
 $\C[[t_p]]$ is the ring of complex-valued functions on the affine scheme  
 $D_{t_p} = {\rm Spec}\;  \C[[t_p]]$ which we call the standard disc $D_{t_p}$.  
As a topological space, $D_p$ can be described by the origin 
  corresponding to the maximal ideal $t_p \; \C[[t_p]]$ and 
 the generic point.      
A morphism from $D$ to an affine scheme
$Z = {\rm Spec}\;  \mathcal R$, where $\mathcal R$ is a $\C$-algebra, 
is a homomorphism of algebras 
$\mathcal R \to \C[[t_p]]$.
  Such a homomorphism can be constructed by realizing $\C[[t_p]]$ as a completion of $\mathcal R$. 
Geometrically, this is an identification of the disc $D_p$ with the formal neighborhood 
of a point on $M$.   
An abstract disc is an affine scheme ${\rm Spec}\;  \mathcal R$, where $\mathcal R$ is a
$\C$-algebra isomorphic to $\C[[t_p]]$. 
On the abstract disc, the maximal ideal $t_p \C[[t_p]]$ has a preferred generator $t_p$. 
In contrast to that, on an abstract disc 
there is no preferred generator in the maximal ideal of ${\mathcal R}$, and there is no 
preferred coordinate. 
Denote by $\Oo_p$ the completion of the local ring of $M$. 
Then $\mathcal O_p$ is non-canonically isomorphic 
to $\Oo = \C[[t_p]]$. To specify such an isomorphism, or equivalently, an isomorphism
between
$D_p = {\rm Spec}\;  \Oo_p$, 
and $D_{t_p} = {\rm Spec}\;  \C[[t_p]]$, 
we need to choose a formal coordinate $t_p$ at $p \in M$, 
i.e., a topological
generator of the maximal ideal $\mathfrak m_p$ of $\Oo_p$.
 In general there is no preferred formal
coordinate at $p\in M$, and $D_p$ is an abstract disc.  
\subsection{Rational functions attached to discs}
To construct a $\widehat{W}$-valued vertex operator algebra bundle on $M/\F$   
 we would like to attach elements of $\widehat{W}$ to  
 the both standard $D_{t_p} = {\rm Spec}\;  \C[[t_p]]$ and 
  abstract discs $D_p$, where $p$ is a point on $M/\F$.   
For a cohomology theory purposes 
 we attach also characteristics of $\widehat{W}$-elements 
represented in terms of rational functions on discs. 
 Let $\mathcal K_x$ be 
  the field of fractions of the ring of integers  
$\mathbb Z$ is the rational field $\mathbb Q$. 
We denote also by 
${\mathcal K}[(t_{p_1}, \ldots, t_{p_s})]$ 
the field of fractions of the polynomial ring  
 over a field ${\mathcal K}$ 
 as the field of rational functions 
 ${\mathcal K}((t_{p_1}, \ldots, t_{p_s}))=
\left\{(R_1(t_{p_1}, \ldots, t_{p_s}))/(R_2(t_{p_1}, \ldots, t_{p_s})):  
 R_1, R_2 \in {\mathcal K}[(t_{p_1}, \ldots, t_{p_s})] \right\}$.  
 For a coordinate $t_p$ on 
$D_p$, there exist isomorphisms $\mathcal O_p = \C[[t_p]]$ and $\mathcal K_p = \C((t_p))$. 
 We denote by $D_p$ and $D^\times_p$ at $p$ the disc and punctured disc 
  defined as ${\rm {\rm Spec}\; } \Oo_p^{(n)}$ and ${\rm {\rm Spec}\; } \; {\mathcal K}_p$)
 correspondingly. 
\subsection{Rational power differentials} 
In this Subsection we recall basic definitions related to differentials \cite{BZF, Sch}.    
 Let $k$ be a rational number. A $k$-differential defined on a manifold $M$ 
is a section of the $k$-th tensor power of the canonical line bundle $\omega$.  
Choosing a local coordinate $t_p$ arround a point $p \in M$  
we may trivialize $\omega^{\otimes k}$ by the non-vanishing 
section $(dt_p)^{\otimes k}$. 
Any section of $\omega^{\otimes k}$ may then be written as $f(t_p)(dt_p)^{\otimes k}$.  
 For another coordinate $\widetilde{t}_p = \rho(t_p)$,   
the same section will be written as $g(\widetilde{t}_p)(d\widetilde{t}_p)^{\otimes k}$, 
where 
$f(t_p) = g(\rho(t_p))(\rho'(t_p))^{\otimes k}$. 
Now let us suppose that we have a section of $\omega^{\otimes k}$ 
 whose representation by a
function does not depend on the choice of local coordinate, i.e., 
$g(\widetilde{t}_p) = f(\widetilde{t}_p)$, and
  $f(t_p) = f(\rho(t_p))(\rho'(t_p))^{\otimes k}$  
 for any change of variable $\rho(t_p)$.  
  We call  
 $f(t_p)(dt_p)^{\otimes k}$ a canonical $k$-differential. 
Let us denote by $\omega_p$ the space of differentials on $D^\times_p$.  
 Given a linear map   
 $\rho : {\mathcal K}_p \to {\rm End} \left(  \widehat{W}_{t_p}  \right)$, 
such that for any $x \in \widehat{W}_{t_p}$ and large enough $l$, we have  
$\rho(\mathfrak m_p)^l \cdot x = 0$,   
 where $\mathfrak m_p$ is the maximal ideal of $\Oo_p$ at $p$.  
Then, according to \cite{BZF},     
 the vertex operator  
$Y(\rho, t_p) = \sum\limits_{s\in \Z}  
\rho(t^s_p)\; t_p^{-s-1} \; dt_p$,   
is a canonical ${\rm End} \left(  \widehat{W}_{t_p}  \right)$-valued differential on $D^\times_{t_p}$, 
  i.e., it is independent of the choice
of coordinate $t_p$. 
\section{The vertex operator algebra bundle on $M/\F$}  
\label{bundle}
In this Section we provide the construction 
of $\widehat{W}$-valued vector bundle $\W_{M/\F}$ on $M/\F$.    
\subsection{Torsors and twists under groups of automorphisms}
 For an admissible $V$-module $W$ we have the filtration  
$W_{t_{p_j}, \le m} =\bigoplus\limits_{i \ge {\rm Re}(\kappa)}^m W_{t_{p_j}, i}$,    
of $W_{t_{p_j}}$ 
 by finite-dimensional ${\rm Aut}_{p_j}\; \Oo^{(n)}$-submodules, $j \ge 1$.    
Suppose $W$ is an admissible vertex operator algebra $V$-module as in the definition given in 
Section \ref{puzo}. 
We now explain how to collect elements of the space 
$\widehat{W}$ 
into an intrinsic object 
on a collection of abstract discs on $M/\F$.    
 We consider a configuration of $l$-points $(p_1, \ldots,  p_l)$ 
on $M/\F$ lying in non-intersecting  
local discs, and we assume that at each point of $(p_1, \ldots, p_l)$  
a coordinate changes independently of changing of coordinates 
on other discs. 
Therefore, the general element of the group of independent automorphisms of coordinates of 
$l$ points on $M/\F$  
${\rm Aut}_l\; \Oo^{(n)}_{p_1, \ldots, p_l}$  
has the form 
$(t_{ p_1}, \ldots, t_{ p_l} ) \mapsto (\rho_1, \ldots, \rho_l )
 (t_{ p_1}, \ldots, t_{ p_l} )$.     

Let us remind the definition of a torsor \cite{BZF}.  
Let $\mathfrak G$ be a group, and $S$ a non-empty set. 
Then $S$ is called a $\mathfrak G$-torsor if it is equipped with a simply transitive 
right action of $\mathfrak G$. 
 For $s_1$, $s_2 \in S$, there exists a unique $\mu \in \mathfrak G$ such that 
$s_1 \cdot \mu = s_2$, 
where the right action is given by 
$s_1 \cdot (\mu \mu') = (s_1 \cdot  \mu) \cdot \mu'$.  
The choice of any $s_1 \in S$ allows us to identify $S$ with $\mathfrak G$ by sending
 $s_1  \cdot \mu$ to $\mu$. 
Applying the definition of a group twist \cite{BZF} to the group  ${\rm  Aut}_l\; \Oo^{(n)}$ and 
 its module $\widehat{W}$ we obtain following the definition.  
Given a ${\rm Aut}_l\; \Oo^{(n)}$-module $\widehat{W}_{(z_1, \ldots, z_l)}$
 and a ${\rm Aut}_l\; \Oo^{(n)}$-torsor $\mathcal X$,   
one defines the $\mathcal X$-twist of $\widehat{W}_{(z_1, \ldots, z_l)}$ as the set 
\[
\V_{\mathcal X} = \widehat{W}_{(z_1, \ldots, z_l)} \; {{}_\times \atop      {}^{  {\rm Aut}_l \; \Oo^{(n)}  }     } \mathcal X  
=  \widehat{W}_{(z_1, \ldots, z_l)} \times  \mathcal X/  \left\{ (w, a \cdot \xi) \sim  (aw, \xi) \right\}. 
\]
for $\xi \in \mathcal X$, $a \in {\rm Aut}_l\; \Oo^{(n)}$, and $w\in \widehat{W}_{(z_1, \ldots, z_l)}$.   
Given $\xi \in \mathcal X$, we may 
 identify $\widehat{W}_{(z_1, \ldots, z_l)}$ with $\V_{\mathcal X}$, 
by $w \mapsto (\xi, w)$.   
This identification depends
on the choice of $\xi$.
 Since ${\rm Aut}_l\; \Oo^{(n)}$ acts on $\widehat{W}_{(z_1, \ldots, z_l)}$ 
by linear operators, the vector space 
structure induced by the above identification does not depend on the choice of $\xi$, 
and $\V_{\mathcal X}$ is canonically a vector space.
If one thinks of ${\mathcal X}$ as a principal ${\rm Aut}_l\; \Oo^{(n)}$-bundle   
over a set of points, 
then $\V_{\mathcal X}$ is simply the associated vector bundle corresponding to 
$\widehat{W}_{(z_1, \ldots, z_l)}$.   
 Any structure on $\widehat{W}_{(z_1, \ldots, z_l)}$ 
(e.g., a bilinear pairing or multiplicative structure) that
is preserved by ${\rm Aut}_l\; \Oo^{(n)}$ will be inherited by $\V_{\mathcal X}$.

Now we wish to attach to any disc a certain twist $\V_{(t_{p_1}, \ldots, t_{p_l} )}$ 
of $\widehat{W}_{(t_{p_1}, \ldots, t_{p_l} )}$, 
so that $\widehat{W}_{(t_{p_1}, \ldots, t_{p_l} )}$
 is attached to the standard discs, 
and for any set of coordinates $(t_{p_1}, \ldots, t_{p_l} )$ on $(D_{p_1}, \ldots, D_{p_l})$ 
we have an isomorphism
\begin{equation}
\label{isomo}
i_{(t_{p_1}, p_1; \ldots; t_{p_l}, p_l)} : 
\widehat{W}_{(t_{p_1},  \ldots, t_{p_1})} \; \widetilde{\to} \;  \V_{(t_{p_1}, \ldots, t_{p_l} )}.
\end{equation}
We then associate 
sections of some bundles on $(D^\times_{t_{p_1}}, \ldots,  D^\times_{t_{p_l}})$  
to elements of $\widehat{W}_{(t_{p_1}, \ldots, t_{p_l} )}$. 
The system of isomorphisms $i_{(t_{p_1}, p_1; \ldots; t_{p_l}, p_l)}$ 
should satisfy certain compatibility condition. 
 Namely, if $(t_{p_1}, \ldots, t_{p_l} )$ and  
$(\widetilde{t}_{p_1}, \ldots, \widetilde{t}_{p_l})$ 
are two sets of coordinates on $(D_{p_1}$, $\ldots$, $D_{p_l})$ 
such that $(\widetilde{t}_{p_1}; \ldots; \widetilde{t}_{p_l}) 
= (\rho_1, \ldots, \rho_l )(t_{p_1}, \ldots, t_{p_l})$, then we obtain 
an automorphism 
$(i^{-1}_{   (\widetilde{t}_{p_1}, p_1, \ldots, \widetilde{t}_{p_l}, p_l)  }
\circ i_{  (t_{p_1}, p_1; \ldots; t_{p_1}, p_1)  }$  
 of $\widehat{W}_{(t_{p_1}, \ldots, t_{p_l} )}$.  
 The condition is that the assignment
$(\rho_1, \ldots, \rho_l) (z_1, \ldots, z_l) 
\mapsto  i^{-1}_{(\widetilde{t}_{p_1}, p_1; \ldots; \widetilde{t}_{p_l}, p_l)}
 \circ i_{(t_{p_1}, p_1; \ldots, t_{p_l}, p_l)}$,  
defines a representation of the group ${\rm Aut}_l\; \Oo^{(n)}$
of independent  
changes of coordinates
 on $\widehat{W}_{(t_{p_1}, \ldots, t_{p_l} )}$. 
 If this condition is satisfied, 
then $\V_{(t_{p_1}, \ldots, t_{p_l} )}$
 is canonically identified with the twist of $\widehat{W}_{(t_{p_1}, \ldots, t_{p_l} )}$ 
 by the 
 ${\rm Aut}_l \; \Oo^{(n)}$-torsor of formal coordinates at $(p_1, \ldots, p_l)$.   

In the next Subsection we will show that given 
 the space $\widehat{W}$      
one can attach to it 
 a vector bundle $\W_{M/\F}$ on the space of leaves $M/\F$ for a foliation $\F$ defined on 
any smooth complex manifold $M$. 
 I.e., the elements of $\widehat{W}$     
give rise to a collection of coordinate-independent 
sections $X(p_1, \ldots, p_l)$ of the bundle $\W^*_{M/\F}$ in the neighborhoods of 
a collection of points $(p_1, \ldots, p_l) \in M/\F$.    
The construction is based on the principal bundle for the group 
${\rm Aut}_l\; \Oo^{(n)}$,  
 which naturally exists on an arbitrary smooth curve and on any collection 
$(D_{p_1}, \ldots, D_{p_l})$ of non-intersecting discs.   
We denote by ${\it Aut}_{(p_1, \ldots, p_l) }$  the set of all 
coordinates $(t_{p_1}, \ldots, t_{p_l})$ on discs $(D_{p_1}, \ldots, D_{p_l})$,   
centered at points $(p_1, \ldots, p_l)$.   
It comes equipped with a natural right action of the group of automorphisms 
 ${\rm Aut}_l \;  \Oo^{(n)}$.     
 If $t_{p_i} \in {\it Aut}_{p_i}$, $1 \le i \le l$, 
and $\rho(z_i) \in {\rm Aut}_i\; \Oo^{(n)}$, then $\rho_i(  t_{p_i} ) \in {\it Aut}_{p_i}$.  
 Furthermore, as it was shown in \cite{BZF} that 
$(\rho_i * \mu_i)( t_{p_i} ) = \mu_i(  \rho_i(t_{p_i})  )$,   
for $1 \le i \le l$, 
it  defines a right simply transitive action of ${\rm Aut}_i\; \Oo^{(n)}$   
 on ${\it Aut}_{p_i}$.  
Thus we see that 
the group ${\rm Aut}_l\; \Oo^{(n)}$ 
 acts naturally on ${\it Aut}_{(p_1, \ldots, p_l)}$,    
 and is a ${\rm Aut}_l\; \Oo^{(n)}$-torsor.   
Thus, we can define the following twist. 
We can introduce the  ${\rm Aut}_l \;  \Oo^{(n)}$-twist of $\overline{W}_{(p_1, \ldots, p_l)}$    
$\V_{(p_1, \ldots, p_l)}= \overline{W}_{(p_1, \ldots, p_l)}   \;  
{  {}_\times \atop      {}^{   {\rm Aut}_l\;  \Oo^{(n)}   }} 
  \;{\it Aut}_{(p_1, \ldots, p_l)}$.    
 The original definition was given in \cite{BD, Wi}.  
For each set of formal coordinates 
$(t_{p_1}, \ldots,  t_{p_l})$ at points $(p_1, \ldots, p_l)$,  
 $(w_1, \ldots, w_l) \in \widehat{W}_{(p_1, \ldots, p_l)}$,   
 any element of the twist $\V_{(p_1, \ldots, p_s)}$     
 may be written uniquely as a
pair 
 $\left((w_1, \ldots, w_l) \right.$, $\left.(t_{p_1}, \ldots, t_{p_l})\right)$.    
\subsection{Definition of $\W_{M/\F}$-bundle of $\widehat{W}$-elements} 
Now let us formulate the definition of fiber bundle associated  
through vectors of elements 
${\bf X} \in \widehat{W}$ defined on any set of standard discs
 $U= (D_{ t_{p_1} }, \ldots, D_{ t_{p_l} })$ 
 around points $( p_1, \ldots, p_l)$  
 on $M/\F$  
 with local coordinates $(t_{p_1}, \ldots, t_{p_l})$.  
We construct an analog of a principal  
${\rm Aut}_l\; \Oo^{(n)}$-bundle for $M/\F$. 
The fiber space is  
provided by vectors ${\bf X}$ of elements 
$X(t_{p_1}, \ldots, t_{p_l})$,   
 given by  
  a fiber bundle $\W_{M/\F}|_{(D_{ t_{p_1}}, \ldots,  D_{ t_{p_l}}  )}$ 
defined by trivializations 
$i_{ (t_{ p_1 }, \ldots,  t_{ p_l }  )}: {\bf X} (p_1, \ldots, p_l)= 
 \left[   X(p_1,  \ldots, p_l)   \right] \to (D_{t_{ p_1 } }, \ldots, D_{t_{ p_l } })$,   
 with a continuous ${\bf X}(p_1, \ldots, p_l)$-preserving 
 right action 
 ${\bf X}(p_1, \ldots, p_l) \times {\rm Aut}_l \; \Oo^{(n)} \to {\bf X}(p_1, \ldots, p_l)$.   
Namely, for two sections $\zeta$,  $\zeta.a$ of $\W_{M/\F}|_{(D_{ t_{p_1} }, \ldots, D_{ t_{p_l} } )}$,    
the map $a \mapsto \zeta.a$ is a homeomorphism for all $a \in {\rm Aut}_l \;  \Oo^{(n)}$.  
Then, according to the definition of a torsor,  
the fiber of such bundle at points $(p_1, \ldots, p_l)$ is the   
${\rm Aut}_l \;  \Oo^{(n)}$-torsor ${\it Aut}_{(p_1, \ldots, p_l)}$.     

Denote by ${\it Aut}_l$ the set of $l$-tuples of local coordinates ${\it Aut}_{(p_1, \ldots, p_l)}$ 
all over leaves of $\F$.   
Given a finite-dimensional ${\rm Aut}_l \; \Oo^{(n)}$-module $\widehat{W}_{i, (p_1, \ldots, p_l)}$,   
 let 
$\W_{M/\F}|_{(D_{ t_{p_1} } , \ldots, D_{ t_{p_l} })}= \widehat{W}_{i, (t_{p_1 }, \ldots,  t_{p_l } )}   
{\times  \atop \; {\rm Aut}_l \;  \Oo^{(n)} } \; {\it Aut}_l$,     
be the fiber bundle associated to $\widehat{W}_{ i, (t_{p_1 }, \ldots,  t_{p_l } )}$ 
 and ${\it Aut}_l$.  
Then, $\W_{M/\F}|_{ (D_{ t_{p_1} }, \ldots, D_{ t_{p_1} } )  }$ is a finite-rank bundle over   
$M/\F|_{ (D_{ t_{p_1} }, \ldots, D_{ t_{p_1} } )}$ whose fiber at a collection of points  
$(p_1, \ldots, p_l) \in M/\F$ is given by the vector $[X(p_1, \ldots, p_l)]$.   
In a vicinity of every point of $(p_1, \ldots, p_l)$ on $M/\F$ we can choose discs  
 $(D_{p_1}, \ldots, D_{p_l})$  such that the bundle $\W_{M/\F}$ over $(D_{p_1}, \ldots, D_{p_l})$ is 
$(D_{p_1}, \ldots, D_{p_l}) \times {\bf X}(p_1, \ldots, p_l)$,  
where ${\bf X}(p_1, \ldots, p_l)$ is a section of $\W_{M/\F}$. 
The fiber bundle $\W_{M/\F}$  
with fiber $\left[X(p_1, \ldots, p_l) \right]$ is a map 
$\W_{M/\F}: \widehat{W} \rightarrow M/\F$  where $M/\F$ is $\W_{M/\F}$-bundle base space.  
 For every set of points $(p_1, \ldots, p_l) \in M/\F$ with local discs  
$(D_{ t_{p_1} }, \ldots, D_{ t_{p_l} } )$   
 $i_{(t_{ p_1}, \ldots, t_{ p_l}  )}^{-1}$ is homeomorphic to 
$(D_{ t_{p_1} }, \ldots, D_{ t_{p_l} } )\times \widehat{W}$. 
Namely, we have for 
$\left[ X(p_1, \ldots, p_l) \right]:
 i_{ (t_{ p_1}, \ldots, t_{ p_l}  ) }^{-1} \rightarrow  (D_{ t_{p_1} }, \ldots, D_{ t_{p_1} } )
  \times \widehat{W}_{t_{p_1}, \ldots, t_{p_l}}$,  
that 
$\mathcal P $ $\circ$ $\left[  X (p_1, \ldots, p_l )    \right]$      
= $i_{   (t_{ p_1}, \ldots, t_{ p_l}  ) }$    
$|_{   i^{-1}_{  (t_{ p_1}, \ldots, t_{ p_l}  )  }    } $
  $ (D_{ t_{p_1} }, \ldots, D_{ t_{p_l} } )$,  
where ${\mathcal P}$ is the projection map on $(D_{ t_{p_1} }, \ldots, D_{ t_{p_l} } )$.    
For an ${\rm Aut}_l\; \Oo^{(n)}$-module $\widehat{W}_{ (t_{ p_1}, \ldots, t_{ p_l}  ) }$  
which has a filtration by finite-dimensional submodules  
$\widehat{W}_{ s, (t_{ p_1}, \ldots, t_{ p_l}  )  }$, $s \ge 0$, we consider the directed 
inductive limit $\W_{M/\F}$ of a system of finite rank bundles $\W_{s, M/\F}$    
 on $M/\F$ defined by embeddings  
 $\W_{s, M/\F} \rightarrow 
\W_{s', M/\F}$, 
for 
 $s \le s'$, i.e., 
  $\W_{M/\F}$ it as a fiber bundle of infinite
rank over $M/\F$.   
\subsection{Explicit construction of canonical intrinsic setup for $\W_M$}
Let $W$ be a quasi-conformal vertex operator algebra $V$-module $W$ defined in Section \ref{puzo}. 
In order to be able to introduce a section $X(p_1, \ldots,  p_l )$ 
of the vertex operator algebra bundle $\W_{M/\F}$ 
 defined on abstract discs $(D^\times_{ p_1}, \ldots, D^\times_{ p_l} )$  
in the coordinate independent 
description, we associate  
$X(p_1, \ldots,  p_l )$ to coordinate independent vector ${\bf X}(v_1, z_1; \ldots; v_l, z_l)$. 

Now let us give the following definition. 
For each set of points $(p_1, \ldots, p_l)$ and elements  
  $(w_1, \ldots, w_l) \in \overline{W}_{z_1, \ldots, z_l}^{\otimes l}$, 
 we define an intrinsic $\widehat{W}$-valued meromorphic    
section $X(p_1, \ldots, p_l)$  
on the punctured discs 
$(D^\times_{ p_1}, \ldots, D^\times_{ p_l} )$ by  
an operation
$(w_1, \ldots, w_l)$, $(p_1, \ldots, p_l)  \mapsto
 X(p_1, \ldots, p_l )$,     
assigning to a vector $X(p_1, \ldots, p_l )$ of $\W_{(D^\times_{ p_1}, \ldots, D^\times_{ p_l} ) }$  
an element 
of ${\bm{ \mathcal K}}_{(p_1, \ldots, p_l)}$ 
 (i.e., rational $\widehat{W}$-valued functions on $(D^\times_{ p_1}, \ldots, D^\times_{ p_l} ) $), 
 defined by the $\W^*_{ (D^\times_{ p_1}, \ldots, D^\times_{ p_l} ) }$-fiber 
 ${\bf X}_{ i_{  (t_{ p_1}, \ldots, t_{ p_l}  )  }  } \in \widehat{W}_{(p_1, \ldots, p_l)}$. 
Consider the operator 
${\rm R} (\rho_1, \ldots, \rho_s)=  
\left[ \widehat \partial_{J} \rho_{i(I)}\right]
=
\left[
\widehat \partial_{J} \rho_{i_1(I)}, 
\widehat \partial_{J } \rho_{i_2(I)}, 
\cdots,   
\\
\widehat \partial_{J} \rho_{i_s(I)}    
\right]^T$.  
The index operator $J$ takes the value of index $z_j$ of arguments in the vector \eqref{norma},  
while the index operator $I$ takes values of index of differentials $dz_i$ in each entry of 
the vector ${\bf X}$ \eqref{overphi0}. 
 Thus, the index operator 
$i(I)=(i_{I}, \ldots, i_s(I))$      
is given by consequent cycling permutations of $I$. 
 We define the operator 
$\widehat\partial_{J } \rho_a =   
\exp(   - \sum_{ {\bf r}_n, \; 
\sum\limits_{i=1}^n r_i  \ge 1 }  
r_J\; \beta^{(a)}_{{ \bf r}_s }\; \zeta^{r_1}_1 \;  \ldots \; 
 \zeta^{r_J}_J \ldots \zeta^{r_s}_s \; \partial_{z_J})$,    
which contains index operators $J$ as index of a  
dummy variable $\zeta_J$ turning into $z_j$, $j=1, \ldots, s$.   
In the last formula $\widehat\partial_{J }$ acts on each argument of maps $X$ in the vector ${\bf X}$.    
In \cite{BZF} it was shown that the mappings   
$(\rho_1, \ldots, \rho_l)(z_1, \ldots,  z_j ) \mapsto R 
\left( \rho_1, \ldots, \rho_l \right)$,   
for $1 \le j \le l$, 
 define a representation of  ${\rm Aut}_l \; \Oo^{(n)}$
on $\widehat{W}_{(z_1, \ldots, z_l)}$ by  
${\rm R} 
\left(\rho \circ \widetilde{\rho}\right) = {\rm R} 
\left(\rho\right) \; {\rm R} 
\left(\widetilde{\rho}\right)$,  
for $\rho$, $\widetilde{\rho} \in {\rm Aut}_l \; \Oo^{(n)}$.   
Then we see that 
 for generic elements   
${\bf X} \left(v_1, z_1; \ldots; v_s, z_s \right) \in \widehat{W}_U$,   
for an admissible vertex operator algebra $V$-module $W$,   
 ${\bf X}(v_1, z_1; \ldots; v_s, z_s)$    
   are independent on  
 changes 
$(z_1, \ldots, z_{s +s'}) \mapsto (\widetilde{z}_1, \ldots,  \widetilde{z}_{s+s'})
= \left((\rho_1, \ldots, \rho_{s+s'})( z_1, \ldots,  z_{s+s'} )\right)$,   
for  $1 \le i \le s+s'$,  
 of local coordinates of  $(z_1, \ldots, z_s)$ and $(\widetilde{z}_1, \ldots, \widetilde{z}_{s'})$,  
 at points $(p_1, \ldots, p_s)$ and $(\widetilde{p}_1, \ldots, \widetilde{p}_{s'})$. 

Indeed, consider the vector 
${\bf X}(v_1, \widehat{z}_1; \ldots; v_s, \widehat{z}_s )$ =  
$ \left[   
X \left(  v_1, \widehat{z}_1 \; d \widehat {z}_{ {\it i}(1)}; \ldots; 
v_s, \widehat{z}_s \; d \widehat {z}_{ {\it i}(s)} \right)  \right]$.     
Note that 
 $d\widehat{z}_j = \sum\limits_{i=1}^n dz_i \;  
{\partial_{z_i} \rho_j}$,
$\partial_{z_i} \rho_j = \frac
{\partial \rho_j} {\partial z_i } $.    
By the definition of the action of ${\rm Aut}_s\; \Oo^{(n)}$,  
when rewriting $d\widehat{z}_i$,
we have  
\begin{eqnarray*}
{\bf X} (v_1, \widehat{z}_1 \; d \widehat{z}_1; \ldots;  
  g_s, \widehat{z}_s \; d \widehat{z}_s)     
&=&  {\rm R}( \rho_1, \ldots, \rho_s) \;  
 \left[  X \left( v_1, z_1 \; d \widehat{z}_{i(1)}; \ldots; v_s, z_s \; d \widehat{z}_{i(s)} \right) 
\right] 
\\
&=&  {\rm R}(\rho_1, \ldots, \rho_s) \;  
\left[  X \left( v_i, z_i \;  
\sum\limits_{j=1}^s \partial_j \rho_{i(s)} \; dz_j  \right) 
\right]. 
\end{eqnarray*}
 By linearity of the mapping $X$, 
 we obtain from the last equation  
\begin{equation}
\label{norma}
{\bf X}(\widehat{v}_1, \widehat{z}_1; \ldots; \widehat{v}_s, \widehat{z}_s)  
=
 {\bf X}(v_1, \widehat{z}_1 \; d\widehat{z}_1; \ldots; v_2, \widehat{z}_s \; d\widehat{z}_s)  
=
 \left[
 X \left( v_1, z_1 \; dz_{i(1)}; \ldots;   v_s, z_s \; dz_{i(s)} \right)   
\right], 
\end{equation}
Due to properties of a vertex operator algebra $V$ admissible module $W$,   
the action of operators $R\left(\rho_1, \ldots, \rho_s\right)$  
 on $(v_1, \ldots,  v_s) \in V^{\otimes s}$ results in a sum of finitely many terms.
 We then conclude that  
 the vector ${\bf X}$ is invariant, 
i.e., 
\begin{eqnarray*}
 {\bf X} (\widehat{v}_1, \widehat{z}_1; \ldots; \widehat{v}_s, \widehat{z}_s) 
&=& {\bf X} \left(v_1, \widehat{z}_1\; d\widehat{z}_1; \ldots;  v_s, \widehat{z}_s\; d\widehat{z}_s\right)  
\\
&=&  {\bf X}\left( v_1, z_1 \;dz_1; \ldots;  v_s, z_s \;dz_s \right)
= {\bf X} (v_1, z_1; \ldots; v_s, z_s).      
\end{eqnarray*}
The insertions of $k$ vertex operators $\left(v_k, t_{ p_k} \; dt_{p_k} \right)$, $k \ge 0$,   
 which are present in the definition of sections of a vertex operator algebra bundle $\W_{M/\F}$,
keep elements ${\bf X}$ invariant with respect to coordinate changes.    
Thus, the construction of spaces $\widehat{W}$ are invariant under the action of the group 
${\rm Aut}_s \; \Oo^{(n)}$.  

 We now formulate the following Lemma used later for the main result  
 Lemma \ref{groupo} of   
this paper for the category of vertex operator algebra bundles $\W_{M/\F}$ on $M/\F$.  
\begin{lemma}
\label{mainpro}
A $\widehat{W}$-valued, independent  
of the choice of coordinates 
$(t_{i, p_j})$, $1 \le i \le n$, $1 \le j \le l$,  
on a set of non-intersecting discs $( D^\times_{i, p_j} )$,    
  section $X(p_1, \ldots, p_n)$    
of the bundle $\W^*_{M/\F}|_{ (D^\times_{ i, p_j }) }$ on   
 the $\overline{W}_{(p_1, \ldots,  p_l)}$-valued fibers   
$X_{ i_{ (t_{p_1}, \ldots, t_{p_l}   ) } }$  defined by \eqref{isomo} on  
$( D^\times_{i, p_j} )$ dual to $\W_{M/\F}|_{  (D^\times_{i, p_j}) }$ is given by the formula 
\begin{equation}
\label{mainfo}
X(p_1, \ldots, p_l)=   
\left[  X_{ i_{ ( t_{p_1}, \ldots,  t_{p_l}  ) }}(w_1, \ldots, w_l)  \right]   
= \left[ X(v_1, z_1; \ldots; v_l, z_l) \right]= {\bf X}(v_1, z_1; \ldots; v_l, z_l), 
\end{equation}
$[X(v_1, z_1; \ldots; v_l, z_l)] \in \widehat{W}_{ ( t_{p_1}, \ldots,  t_{p_l}  )}$,   
 where $( t_{p_1}, \ldots, t_{p_l} )$ are coordinates 
on the discs $(D^\times_{p_1}, \ldots, D^\times_{p_l})$, 
and $(w_1, \ldots, w_l) \in \overline{W}_{(z_1, \ldots, z_l)}$.  
\end{lemma}
\begin{proof}
Now let us proceed with the explicit construction of ${\bf X}_{ i_{   (t_{p_1}, \ldots, t_{p_l})  }}$.  
By choosing coordinates $( t_{p_1}, \ldots, t_{p_l} )$ on a collection of discs 
$(D^\times_{ p_1}, \ldots, D^\times_{ p_l} )$,  
 we obtain a trivialization  
$i_{(t_{p_1}, \ldots, t_{p_l})   }: {\bf X} \left( \widehat{W} [[  (t_{p_1}, \ldots, t_{p_l} )   ]] \right) 
 \; {\widetilde{} \atop \rightarrow} \;  
\Gamma \left( \W_{M/\F}|_{   (D^\times_{ p_1}, \ldots, D^\times_{ p_l} )   } \right)$, 
 of the bundle $\W_{ (D^\times_{ p_1}, \ldots, D^\times_{ p_l} )}$
which we call the $( t_{p_1}, \ldots, t_{p_l} )$-trivialization. 

We also obtain trivializations of the fiber 
 $\widehat{W}_{(p_1, \ldots, p_l) } \; {\widetilde{}  \atop \rightarrow }  
\; \gamma \left( \W_M|_{ (D^\times_{p_1}, \ldots, D^\times_{p_l})} \right)$, 
 and its dual
$\widehat{W}^*_{(p_1, \ldots, p_l) }  \; {\widetilde{}  \atop \rightarrow} \;   
\gamma\left( \W^*_{M/\F}|_{   (D^\times_{p_1}, \ldots, D^\times_{p_l}) }    \right)$. 
Let us denote by $\left( w_1,  t_{ p_1}; \ldots; w_1,  t_{ p_l}   \right)$  the image  
of $(w_1, \ldots, w_l) \in \widehat{W}_{z_1, \ldots, z_l}$ 
in $\W_{M/\F}|_{ (D^\times_{p_1}, \ldots, D^\times_{p_l}) } $  
and by $\left( t_{ p_1}; \ldots;  t_{ p_l}    \right)$ of  
$\W^*_{M/\F}|_{  ( D^\times_{ t_{p_1} }, \ldots,  D^\times_{ t_{p_l} }  )    } $    
under $(t_{ p_1}, \ldots,   t_{ p_l})$-trivialization. 
In order to define the required section ${\bf X}( p_1, \ldots, p_l )$ 
with respect to these trivializations     
we need to attach an element of $(\widehat{W}_{(  t_{ p_1}, \ldots ,  t_{ p_l}  ) }$ 
 to each $\left( v_1, t_{p_1 }; \ldots;  v_l, t_{p_l } \right)  
\in \W_{M/\F}|_{  (D^\times_{ t_{p_1} }, \ldots,  D^\times_{ t_{p_l} })       }$,  
and a section $i_{ (t_{ p_1}, \ldots,   t_{ p_l}) }(x_1, \ldots, x_l)$  
 of $\W|_{  (D^\times_{ t_{ p_1 } }, \ldots,  D^\times_{ t_{ p_l } } )  }$ 
for $(x_1, \ldots, x_l) \in     \widehat{W}_{ (t_{p_1}, \ldots, t_{p_l}  )} $. 
 It is sufficient to assign a function to the sets  
$(v_1, z_1; \ldots; v_l, z_l)$,  $(w_1, \ldots, w_l) \in \widehat{W}_{  (z_1, \ldots, z_l)  }$ 
in the $(t_{p_1}, \ldots, t_{p_l} )$-trivialization. 
Thus, we identify a $\widehat{W}$-valued  
section $\widetilde{\bf X}  (p_1, \ldots, p_l)$ of  
$\W^*_{  ( D^\times_{p_1}, \ldots, D^\times_{p_l}  )}$,  
with the section ${\bf X}( v_1, z_1; \ldots; v_l, z_l)$ of 
$\W_{   ( D^\times_{p_1}, \ldots, D^\times_{p_l}  ) }$ by means 
of formula \eqref{mainfo}. 

Let $(\widetilde{ t}_{p_1 }, \ldots, \widetilde{ t}_{p_l} ) =  
   (\rho_1, \ldots, \rho_l  ) (  t_{p_1 }, \ldots,  t_{p_l}  )$ be another set of coordinates. 
 Then, using the above arguments, we construct analogously
a section $\widetilde {\bf X} ( p_1, \ldots, p_l )$ by the formula 
\[
  \widetilde{\bf X}  (p_1, \ldots, p_l)= \left[ 
 \widetilde{X}_{ i_{  ( \widetilde{t}_{p_1}, \ldots, \widetilde{t}_{p_l} ) }} (w_1, \ldots, w_l  )
   \right]   
= \left[  X( \widetilde{v}_1, \widetilde{z}_1; \ldots; \widetilde{v}_l, \widetilde{z}_l)   \right]=
 {\bf X}( \widetilde{v}_1, \widetilde{z}_1; \ldots; \widetilde{v}_l, \widetilde{z}_l).  
\]
 Since $\left( i^{-1}_{ ( t_{p_1}, \ldots, t_{p_l}  ) }  
\circ i_{  ( t_{p_1}, \ldots, t_{p_l}  )}   \right)$  
is an automorphism of $\widehat{W}_{ (p_1, \ldots, p_l) }$, we represent a 
   change of variables $\widetilde{t}_{ p_j } = \rho_j(z_j)$, $1 \le j \le l$,  
in terms of composition of trivializations 
\begin{equation}
\label{kozel}
\rho_j (z_j) \mapsto  i^{-1}_{ \widetilde{t}_{j, p}  } \circ i_{ t_{j, p} }, 
\end{equation}
and, therefore, 
relate ${\bf X}_{i_{   (\widetilde{ t  }_{p_1 }, \ldots,  \widetilde{ t  }_{p_l }     )  }}
(\widetilde{w}_1, \ldots, \widetilde {w}_l)$  
with ${\bf X}( {i_{ ( t_{p_1}, \ldots, t_{p_l}  ) } } (w_1, \ldots, w_l)$. 
Since \eqref{kozel}  
defines a representation on $\widehat{W}$ of the group ${\rm Aut}_l\; \Oo^{(n)}$ of changes of coordinates, 
 then $\widehat{W}_{(p_1, \ldots, p_l)}$ is canonically identified with the twist of $\overline{W}$ by the 
 ${\rm Aut}_l \; \Oo^{(n)}$-torsor of formal coordinates at $(p_1, \ldots, p_l)$.  
Using definition of a torsor one sees that   
 elements of the space $\widehat{W}^q_r|U$  
can be treated as 
${\rm Aut}_l\; \Oo^{(n)}$-torsor    
of the product of groups of a coordinate transformation, namely,    
that 
$(v_1, z_1; \ldots; v_l, z_l)= \left( R(\rho_l)^{-1}. 
(v_1,  \widetilde{t}_{ p_1 }; \ldots; v_l,  \widetilde{t}_{ p_l} \right)$, 
Thus, we relate the l.h.s and r.h.s. of \eqref{mainfo}.  
 Since the element ${\bf X}(v_1, z_1; \ldots; v_l, z_l)$ is invariant 
with respect to changes of coordinates, Lemma follows.  
\end{proof}
\section{Category of vertex algebra bundles on leaves of $M/\F$ and transversal sections}  
\label{cosimplicial}
In this Section we construct canonical twisted vertex algebra $V$-module bundle $\W_{M/\F}$ 
on leaves and transversal sections of a 
codimension $p$ foliation $\F$ defined on a smooth $n$-dimensional manifold $M$.  
\subsection{Holonomy and transversal basis for a foliation}
\label{holonomy}
Let us first recall \cite{CM} definitions of transversal  
basis and holonomy embeddings for a foliation $\F$.
Transversal sections $U_i$ of a foliation $\F$ passing through points 
$p_i$, $i \ge 0$, are neighborhoods 
of the leaves through $p_i$ in the leaf space $M/\F$.  
Suppose we are given a path $\alpha$ between two points 
$p_1$ and $p_2$ which belong to the same leaf of $\F$.  
For two transversal sections $U_1$ and $U_2$ passing through 
$p_1$ and $p_2$ one defines a transport $\alpha$  
 along the leaves from a neighborhood of $p_1 \in U_1$ to a neighborhood of $p_2 \in U_2$. 
Then it is assumed that there exists a germ of a diffeomorphism
 ${\rm hol}(\alpha): (U_1, p_1)\rmap (U_2, p_2)$   
called the holonomy of $\alpha$. 
When the transport 
  $\alpha$ is defined in all of $U_1$ and embeds into $U_2$
then 
$h: U_1\hookrightarrow U_2$ is denoted by 
${\rm hol}(\alpha): U_1\hookrightarrow U_2$ and it is called a holonomy embedding.   
A composition of paths induces a composition of corresponding holonomy embeddings.  
 Two homotopic paths always define the same holonomy.
 The holonomy groupoid 
\cite{Co, Haefl, Wi} 
is the groupoid $Hol(M, \F)$ over $M$ where arrows $p_1 \rmap p_2$ are such germs $hol(\alpha)$. 
A transversal basis $\U$ for $\F$ is a set of transversal sections $U_i\subset M$     
such that for a section $U_i$ passing through a point $p_i$, 
and for any transversal section $U_j$ passing through 
 $p_j \in M$, one can find a 
holonomy embedding $h: U_i\hookrightarrow U_j$ with $U_i\in \U$ and $p_j\in h(U_i)$.
\subsection{Spaces of sections of $V$-bundles}
Let $M$ be endowed with a coordinate chart   
$\V=\left\{V_r, r\in \Z \right\}$. 
Consider a (possibly infinite) arbitrary set $p_l$, $l \ge 0$, of $l$ distinct points 
and corresponding domains $V_l \subset M$.      
Let $\U=\left\{ U_k \right\}$, $k \ge 0$, be a transversal basis of $\F$. 
 We chose a (possibly infinite) set  $(p_{l+1}, \ldots, p_{l+m})$ 
of arbitrary distinct $m \ge 0$ points  
on a set of sections $U_{i', b}$, $1 \le i' \le m$, $1 \le b \le k$ of $\U$.  
Let us associate to each point of $(p_1, \ldots, p_l)$ and $(p_{l+1}, \ldots, p_{l+m})$  
 vertex operator algebra elements  
 $(v_{i, 1}, \ldots, v_{i, n})$, $1 \le i \le l$, 
 and $(v_{j, 1}, \ldots, v_{j, p})$, $l+1 \le j \le l+m$ correspondingly.  
For the set of all choices of $ln+mp$ vertex operator algebra $V$ elements and 
  $\widetilde{n}=ln+mp$,  it is convenient to 
  renotate the vertex operator algebra elements as
\[
(\widetilde{v}_1, \ldots, \widetilde{v}_{\widetilde{n}} ) =  
 (v_{1, 1}, \ldots, v_{l, 1}, \ldots,  v_{1, n}, \ldots, v_{l, n}, 
v_{l+1, }, \ldots, v_{l+1, p}, \ldots,  v_{l+m, 1}, \ldots, v_{l+m, p}).
\] 
Endow each of points among $(p_1, \ldots, p_l)$ and $(p_{l+1}, \ldots, p_{l+m})$ with  
  sets 
$(z_{i, 1}, \ldots, z_{i, n})$, $1 \le i \le l$
 and $(z_{j, 1}, \ldots, z_{j, p})$, $l+1 \le j \le l+m$, of  
 local coordinates on domains $V_l$ and $V_{l+m}$. 
 Denote also 
\[
(\widetilde{z}_1, \ldots, \widetilde{z}_{\widetilde{n}} ) =   
 (z_{1, 1}, \ldots, z_{l, 1}, \ldots,  z_{1, n}, \ldots, z_{l, n}, 
z_{l+1, }, \ldots, z_{l+1, p}, \ldots,  z_{l+m, 1}, \ldots, z_{l+m, p}). 
\]

Now, taking into account the content of Section \ref{bundle}, and, in particular, 
 Lemma \ref{mainpro}, 
 we are on a position to introduce the spaces of sections of vertex operator algebra $V$ bundle $\W_{M/\F}$
 over leaves and transversal sections of a codimension $p$ foliation $\F$ defined on $M$. 
Note that the space of $\F$-leaves is not in general a manifold.  
Nevertheless, one can always consider local coordinates in appropriate domains on leaves of $M/\F$  
induced by local coordinates on a chart defined on $M$. 
 In this Section we provide the specific form of canonical sections ${\bf X}$ of a 
vertex operator algebra bundle $\W_{M/\F}$  
as elements of the spaces $\overline{W}$ considered on specific domains on $M/\F$ and $\U$.  
In order to work with objects having coordinate invariant formulation, we consider 
elements of $\overline{W}$ with local coordinates $z$ multiplied by 
powers of corresponding differentials $dz$. 
For all choices of $l$ points and all choices of 
vertex operator algebra elements  
  for $ln \ge 0$ complex variables $(\widetilde{z}_1, \ldots, \widetilde{z}_{ln})$ defined in domains 
$V_k$, $1 \le k \le ln$ of the coordinate chart $\V$ on $M$, 
 let us consider the vector of the form \eqref{overphi0} with variables 
 $(\widetilde{v}_1, \widetilde{z}_1; \ldots;  \widetilde{v}_{ln}, \widetilde{z}_{ln})$, 
containing $\overline{W}$-rational functions $X$.           
In \cite{BZF}, in the case $n=1$, they proved for primary $u$ that the vertex operator 
$Y_W(u, z) \; dz^{\wt(u)}$ is an invariant object with respect to 
changes of the local coordinate. 
In previous Section we proved that the vectors ${\bf X}$ introduced above 
 as well as vertex operators  
  $\mathcal Y(u, z_i)= Y_W(u, z_i)\; dz_i^{\wt(u)}$, $i \ge 0$, for primary $u \in V$,    
are invariant with respect 
to changes of coordinates, 
i.e., to the group of coordinate transformations ${\rm Aut}_s\; \Oo^{(n)}$ on $M/\F$    
$\left(w_1, \ldots, w_s\right) \mapsto (z_1, \ldots, z_s)$, and corresponding differentials.  

In \cite{GF} the classical approach to cohomology of vector fields of manifolds 
was initiated. In \cite{Fei, Wag} we find an alternative way to describe 
  cohomology of Lie algebra of vector fields on a manifold in the cosimplicial setup.  
Taking into account the standard methods of defining canonical (i.e., independent of the choice  
of covering $\U$ and coordinates) cosimplicial object \cite{Fei, Wag}   
as well as the ${\rm \check{C}}$ech-de Rham 
cohomology construction \cite{CM}, we formulate here the vertex operator algebra 
approach to cohomology of a foliation.   
Let $I_q =\left\{z_{i_s, j_s} \right\}$, $1 \le s \le q$, 
  be a subset (with no repetitions) of the set 
$(z_{i, 1}, \ldots, z_{i, n})$, $1 \le i \le l$
 of local variables 
corresponding to 
of $l$ points $(p_1, \ldots, p_l)$ taken on the same leaf $\mathfrak f$ of $M/\F$.  
Similarly, let $J_r=\left\{z_{i_{s'}, j_{s'}}\right\}$, $1 \le s' \le r$, 
be a subset (with no repetitions) of the set 
$(z_{i, 1}, \ldots, z_{i, p})$, $l+1 \le i \le l+m$, 
 of local variables 
corresponding to 
of $m$ points $(p_{l+1}, \ldots, p_{l+m})$ on sections of a transversal basis $\U$ of $\F$.  

Assume that all points $(p_1, \ldots, p_{l+m})$ belong to the same leaf $\mathfrak f$ of $M/\F$.  
 Consider a subspace   $\widehat{W}^q(V, W, U, \F)$  of vectors
 ${\bf X}(v_{i_1, j_1}, z_{i_1, j_1}; \ldots; v_{i_q, j_q}, z_{i_q, j_q})$
associated to $I_q$ for all sets of vertex operator algebra elements 
$(v_{i_1, j_1}, \ldots, v_{i_q, j_q})$ with local coordinates defined on a domain $U$ of $M$.  
A vertex operator algebra $V$ bundle $\W_{M/\F}$ consists of the 
union of the spaces 
$\widehat{W}^q_r=\widehat{W}^q_r(V, W, \F)$, $q$, $r \ge 0$. 
The spaces $\widehat{W}^q_r$ consist of  
of sections of $\W_{M/\F}$ defined  
as the space $\widehat{W}^q(V, W, \F)$ of vectors ${\bf X}$ 
on each leaf $\mathfrak f$ of $\F$, 
and any subset $J_r$ of $r$  
vertex operators for all sets of vertex operator algebra elements $v_{i'_a, j'_a}$
${\mathcal Y}_W (v_{i'_a, j'_a}, z_{i'_a, j'_a})$, $1 \le a \le r$,  
with local coordinates $z_{i'_a, j'_a}$ defined in $r$ subdomains $U_{i'_a, b} \subset U_b$,  
 $1 \le b \le k$ of $k$ transversal sections $U_b$ of a transversal basis $\U$. 
Here the domain $U$ is defined as 
$U= \bigcap_{ {U_{i'_1, 1} \stackrel{h_{i'_1, 1}}{\rmap}  \ldots \stackrel{h_{i'_r, k}}{\rmap} U_{i'_r, k}, 
1 \le a \le r, \; 1 \le b \le k,
} \; } 
U_{ i'_a, b}$, 
where the intersection ranges over $r$ subdomains $U_{i'_a, b} \subset U_b \subset \U$ of 
$r$ local coordinates for any choice of $l$ points $(p_1, \ldots, p_l)$  
on the same leaf $\mathfrak f$ of $M/\F$ 
related by the holonomy embeddings $h_{i'_a, b}$, $1 \le a \le r$,  $1 \le b \le k$.  
In the case $k=0$ the sequence of holonomy embeddings is empty.  
It is easy to see that the definition of $\widehat{W}^q_r$  
does not depend on the choice of $\U$. 

The spaces $\widehat{W}^q_r$ are related by the shift operators 
  $\Delta^q_r: \widehat{W}^q_r \to \widehat{W}^{q+1}_{r-1}$, 
 increasing the upper index and decreasing the lower index in elements of $\widehat{W}^q_r$.   
For $I_{q+1}=(i_k, j_k)$, $1 \le k \le q+1$, and $J_r=(i'_{k'}, j'_{k'})$, $1 \le k' \le r$, 
and ${\bf X} \in \widehat{W}_r^q$    
let us define  
the operator in the standard way \cite{CM, Huang} 
\begin{eqnarray}
\label{hatdelta}
&& \Delta^q_r {\bf X} (v_{i_1, j_1}, z_{i_1, j_1}; 
\ldots; v_{i_q, j_q}, z_{i_q, j_q})    
\nn
&& \qquad 
 =  \mathcal Y_W (v_{i_1, j_1}, z_{i_1, j_1})    
 {\bf X}\left(v_{i_2, j_2}, z_{i_2, j_2}; \ldots; v_{i_q, j_q}, z_{i_q, j_q} 
\right)     
\nn
&& \quad +\sum_{s=1}^q(-1)^s  
 {\bf X} \left(  
v_{i_1, j_1}, z_{i_1, j_1};  \ldots; 
\mathcal Y_W (v_{i_{s-1}, j_{s-1}}, z_{i_{s-1}, j_{s-1}} -\zeta_s)  \right. 
\nn
&& \quad \qquad \left. 
\mathcal Y_W (v_{i_{s+1}, j_{s+1}}, z_{i_{s+1}, j_{s+1}}- \zeta_s) \one_V 
      ; \ldots; v_{i_q, j_q}, z_{i_q, j_q}  
\right)       
\nn
 && \qquad  +(-1)^{q+1}  
  \mathcal Y_W( v_{i_{q+1}, j_{q+1} }, z_{i_{q+1}, j_{q+1}} ) 
{\bf X} \left( v_{i_1, j_1}, z_{i_1, j_1}; \ldots;  v_{i_q, j_q}, z_{i_q, j_q} \right).  
\end{eqnarray}
The shift operator $\Delta^q_r$ is chosen in such a way that its characteristics of 
 would have nice analytic and cohomological properties.    
Note that, after application of $\Delta^q_r$ on an element ${\bf X}$ containing 
local coordinates and corresponding vertex operator algebra elements associated to all $n$ dimensions,  
of $M$, the result is of such action is then related to submanifold with less number of 
local coordinates describing points $(p_1, \ldots, p_l)$. 
For $q=2$, there exists a subspace $\widehat{W}_{\vartheta }^2$ 
of $\widehat{W}_0^2$     
containing $\widehat{W}_r^2$ for all $r \ge 1$ such that     
$\Delta^2_r$ is defined on this subspace.  
 For $J_3=(i_k, j_k)$, $1 \le k \le 3$,  the operator $\Delta^2_{\vartheta}$   
is defined for ${\bf X}\in \widehat{W}_{\vartheta}^2$ by a particular case of \eqref{hatdelta}.   
With the shift operator $\Delta^q_r$ we obtain the sequences: 
$\widehat{W}_r^0
\stackrel{ \Delta^0_r }{\longrightarrow}
\widehat{W}_{r-1}^1
\stackrel{\Delta^1_{r-1} }{\longrightarrow} 
(\widehat{W}_{r-2}^2, \delta_{r, 3} \; \widehat{W}^2_\vartheta ) 
 \stackrel{(\Delta^{r-2}_2, \; \delta_{r, 3} \; \Delta^2_\vartheta )}{\longrightarrow} 
\cdots 
\stackrel{\Delta^{r-1}_1}{\longrightarrow} 
\widehat{W}_0^r$. 

The construction of the vertex operator algebra $V$-bundle $\W_{M/\F}$ provides a description of 
the holonomy groupoid $Hom(M, \F)$ introduced in Subsection \ref{holonomy} in terms of holonomy embeddings.  
We consider the spaces of vectors ${\bf X} (h_{i_1, 1}, \ldots, h_{i_r, k}) ={\bf X}|_{U}$ taken on 
all leaves $\mathfrak f$ of $M/\F$.    
In terms of holonomy embedding, 
the shift operator (c.f. \cite{CM}) in this case is given by the standard differentials. 
The vertical differential
 $\widehat{W}^q_r \rmap \widehat{W}^q_{r+1}$ is $(-1)^q d^{-1}$ where $d$ 
is the usual De Rham differential.
The horizontal differential $\Delta: \widehat{W}^q_r \rmap \widehat{W}^{q+1}_r$   
is $\Delta = \sum(-1)^i\Delta_i$,  
where
\begin{eqnarray}
\label{deltasy} 
\nonumber 
\Delta_i {\bf X} (h_1, \ldots , h_{q+1}) &=&    
\delta_{i, 0} \; h_1{}^* \; {\bf X}(h_2, \ldots , h_{q+1})     
 + \delta_{i, q+1} {\bf X}(h_1, \ldots, h_q)  
\\
&+&(1-\delta_{i, 0}) (1-\delta_{i, q+1}) {\bf X}(h_1, \ldots, h_{i+1}h_i, \ldots, h_{k+1}).  
\end{eqnarray}
The category $\mathfrak W_{M/\F}$ of vertex operator algebra admissible $V$-bundles $\W_{M/\F}$   
for a foliation $\F$ consists of objects $\W_{M/\F}$ with      
morphisms provided by intertwining operators \cite{DL}.  
\subsection{Characteristics of bundle $\W_{M/\F}$-sections} 
\label{tuzla}
In the definition of spaces $\widehat{W}$    
  sequences of holonomy embeddings $h_i$, $i \ge 0$ were involved. 
For germs $hol(\alpha)$ of the groupoid $Hol(M, \F)$ over $M$   
we define the spaces $\widetilde{X}$ of vectors ${\bf X}$ defined in previous subsections. 
The holonomy groupoid 
is the groupoid $Hol(M, \F)$ over $M$ where arrows $p_1 \rmap p_2$ are such germs $hol(\alpha)$. 
In this Section we prove the main result of this paper for the category of  
$V$-bundles for foliations defined on a complex manifold. 
Recall that the cohomology of $Hol(V, \F)$ determines the cohomology of a foliation 
$\F$ \cite{CM}. 
The main advantage of Lemma \ref{groupo} provided at the end of this Section, 
is that by using vertex operator algebra properties  
we are able to compute explicitly the cohomology of the holonomy groupoid 
$Hol(M, \F)$ in terms of special functions. 
For meromorphic functions 
 of several complex variables  
 defined on sets of open  
 domains of $M$ with local coordinates $z_{i, j}$     
 which are extendable rational functions $f(z_{i, j})$ 
 on larger domains on $M$  we denote such extensions by $R(f(z_{i, j} ))$.   
For a set of $\overline{W}_{ ( z_{i,j} ) }$-defining elements $(v_{i,j})$     
we consider the converging 
rational 
functions $f(v_{i,j}, z_{i,j})\in \overline{W}_{ ( z_{i,j} ) }$ of $ z_{i,j} \in F_{ln}\mathbb C$. 

By involving the definition of $\widehat{W}^q_r$ it is possible to introduce 
a vertex operator algebra $V$  
 cohomology of the leaves space $M/\F$ of a foliation $\F$.  
Let us consider the spaces $C^q_r=C^q_r(V, W, \F)$ containing vectors of 
rational functions    
 provided by vectors of characteristics $ [ \Omega(X) ]$ of ${\bf X}$-entries.      
For any ${\bf X} \in \widehat{W}^q_r$,  
the map $\Delta^q_r$ induces the map $\delta^q_r$ by $[\Omega(X)]$.  
The coboundary operator $\delta^q_r$ exhibits the chain-cochain property if characteristics of 
entries $X$ of ${\bf X}$ satisfy the following conditions. 

For sets $p_j$, $1 \le j \le l$ of $l$ points on the same leaf of $M/\F$, 
 we consider a map   
  $X(v_{i,1}, z_{i, 1};  \ldots; v_{i,n}, z_{i, n}): 
V^{\otimes ln} \in \overline{W}[[z_1, \ldots, z_r ]]$,
 $1 \le i \le l$,
combined with a set of $mp$ vertex operators at points $p_{l+k}$, $1 \le k \le m$, 
such that its characteristic $\Omega(X)$ satisfy the following properties. 
We imply certain conditions on the characteristics 
$\Omega(v_1, z_1; \ldots; v_s, z_s)=\Omega(X(v_1, z_1; \ldots; v_s, z_s))$ for 
 elements $X(v_1, z_1; \ldots; v_s, z_s)$
by to the relations mentioned below
  to be coherent with definitions for $\overline{W}$ given in \cite{Huang}.  
We require that for $i=1, \ldots, s$, 
\[
  \partial_{z_i}   
\Omega(v_1, z_1; \ldots; v_s, z_s)  
 =   
 \Omega(v_1, z_1; \ldots; v_{i-1}, z_{i-1}; L_V(-1) v_i, z_i;   
 v_{i+1}, z_{i+1}; \ldots; v_s, z_s),
\]        
\[
 \left( \partial_{z_1} + \cdots + \partial_{z_s} \right) 
 \Omega(v_1, z_1; \ldots; v_s, z_s) 
=  L_W(-1) 
\Omega(v_1, z_1;  \ldots; v_s, z_s). 
\]
Since $L_W(-1)$ is a weight-one operator on $W$,  
$e^{zL_W(-1)}$ is a linear 
operator on $\overline{W}$ for any $z\in \C$. 
For a linear map  $X$  
 with $(v_1, \ldots, v_s) \in V^{\otimes s}$,  
  $(z_1, \ldots, z_s) \in F_s \C$,  $z\in \C$ such that 
$(z_1+z, \dots,  z_s+ z)\in F_s\C$, 
the characteristics   
\[
  e^{zL_W(-1)}  
\Omega(v_1, z_1; \ldots; v_s, z_s)    
= 
\Omega(v_1, z_1+z; \ldots; v_s, z_s+z),
\]     
and for $(v_1, \dots, v_s)\in V^{\otimes s}$,
 $(z_1, \ldots, z_s)\in F_s \C$, $z \in \C$ and $1\le i\le s$ such that  
$(z_1, \ldots, z_{i-1}, z_i+z, z_{i+1}, \dots, z_s)\in F_s\C$, 
the power series expansion of the characteristic of the element 
$\Omega(v_1, z_1; \ldots; v_{i-1}, z_{i-1};  v_i, z_{i}+z; v_{i+1}, z_{i+1}; \ldots;  v_s, z_s)$,      
in $z$ are equal to the power series expansion of the characteristic  
$\Omega(v_1, z_1; \ldots; v_{i-1}, z_{i-1}; e^{zL_V(-1)}v_i, z_i;$ 
 $v_{i+1}, z_{i+1}; \ldots; v_s, z_s)$,      
in $z$.
In particular, the power series in $z$ is absolutely convergent
  on the open disc $|z|<\min_{i\ne j}\{|z_i-z_j|\}$. 
In addition to that,  
for $(v_1, \dots, v_s) \in V^{\otimes s}$,  
  $(z_1, \ldots, z_s) \in F_s\C$ and $z\in \C^\times$ so that  
$(zz_1, \ldots, zz_s)\in F_s \C$,
a linear map $X: V^{\otimes s}\to  
\overline{W}$  
 the characteristics   
\[
  z^{L_W(0)}  
 \Omega(v_1, z_1; \ldots; v_s, z_s) 
 = 
 \Omega(z^{L_V(0)} v_1, zz_1; \ldots; z^{L_V(0)} v_s, zz_s).
\]   
should coincide. 

 Recall now the definition of shuffles.   
For $l > 0$ and $1\le s \le l-1$, let $J_{l; s}$ be the set of elements of 
$S_l$ which preserve the order of the first $s$ and the last 
$l-s$ numbers,  i. e., 
$J_{l, s}=\{\sigma\in S_l \;|\;\sigma(1)<\ldots <\sigma(s),\;
\sigma(s+1)<\ldots <\sigma(l)\}$. 
The elements of $J_{l; s}$ are then called shuffles. 
We will use the notation $J_{l; s}^{-1}=\{\sigma\;|\;
\sigma\in J_{l; s}\}$ for them.
Finally, define the left action of the permutation group $S_r$ on $\overline{W}$ by 
$\sigma(f)(z_1, \ldots, z_r)=f(z_{\sigma(1)}, \ldots, z_{\sigma(r)})$,     
 for $f\in \overline{W}$.    
We require that  
\[
\sum_{\sigma\in J_{l; s}^{-1}}(-1)^{|\sigma|} 
\sigma(\Omega 
(v_{\sigma({1,1})}, z_{\sigma({1,1})}; \ldots; v_{\sigma({l,1})}, 
\ldots, v_{\sigma({1,n})}, \ldots, v_{\sigma(l,n)})
)=0.
\] 

Denote by $P_s: \overline{W} \to \overline{W}_{(s)}$, 
 the projection of $\overline{W}$ on $\overline{W}_{(s)}$. 
Denote by $(l_i, \ldots, l_{\widetilde{n}})$ a partition of ${\widetilde{n}}$    
of $\widetilde{n} =\sum_{i \ge 1} l_i$, $k_i=l_{1}+\cdots +l_{i-1}$,  
and $\zeta_i \in \C$. 
Consider the local coordinates $(\widetilde{z}_{ln+1}, \ldots, \widetilde{z}_{\widetilde{n}})$ 
of points $(p_{l+1}, \ldots, p_{l+m})$ bounded    
 in the domains  
$|\widetilde{z}_{k_i+ k'} -\zeta_i| + |\widetilde{z}_{k_j+k''}-\zeta_j|< |\zeta_i -\zeta_j|$,   
for $i$, $j=1, \dots, n$, $i\ne j$, and for $k'=1, \dots$,  $l_i$, $k''=1$, $\dots$, $l_j$.   
For $\widetilde{k}_i= k_i+l_i$, define  
$f_i=\Omega \left( 
  \mathcal Y_W ( \widetilde{v}_1,  \widetilde{z}_1 - \zeta_i )   
 \ldots  
 \mathcal Y_W (  \widetilde{v}_{ \widetilde{k}_i }, \widetilde{z}_{\widetilde{k}_i} - \zeta_i )  
\right)$, 
for $i=1, \dots, ln$. 

Assume that there exist positive integers $\beta(\widetilde{v}_{l', i}, \widetilde{v}_{l", j})$ 
depending only on $\widetilde{v}_{l', i}$ and $\widetilde{v}_{l'', j}$  for 
$i$, $j=1, \dots, \widetilde{n}$, $i\ne j$, $ 1 \le l', l'' \le \widetilde{n}$, 
such that the characteristic 
\begin{equation}
\label{pairing1}
 \Omega \left( \sum\limits_{ r_1, \ldots, r_{ \widetilde{n} } \in \Z}    
  X \left(P_{r_1}  f_1, \zeta_1;  \ldots;  
  P_{r_{\widetilde{n}}}  f_{\widetilde{n}},  \zeta_{\widetilde{n}}  \right) \right), 
\end{equation}   
is absolutely convergent in the domains defined above to an analytic extension 
in $(\widetilde{z}_1, \ldots, \widetilde{z}_{\widetilde{n}})$ 
independently of complex parameters $(\zeta_1, \ldots, \zeta_{\widetilde{n}})$,
with poles of order less than or equal to $\beta(\widetilde{v}_{l',i}, \widetilde{v}_{l'', j})$
allowed only on the diagonal of  
$(\widetilde{z}_1, \ldots, \widetilde{z}_{\widetilde{n}})$. 
We assume that for $(\widetilde{v}_1, \ldots, \widetilde{v}_{\widetilde{n}})$,  
the characteristic  
\begin{equation}
\label{pairing2}
 \Omega\left(\sum_{q\in \C} \mathcal Y_W (v_{l+1, 1}, z_{l+1, 1}) \ldots   
 \mathcal Y_W (v_{l+m, 1}, z_{l+m, p})   
  P_q X\left( v_{1,1}, z_{1, 1}; \ldots; v_{l, 1}, z_{l, n} \right) \right),
\end{equation}    
 (incorporating local coordinates on $M$ and transversal sections)  
is absolutely convergent 
on the domains   
$|\widetilde{z}_i|>|\widetilde{z}_s|>0$, for $i=1, \dots, m$, and  
$s=m+1, \dots, m+l$,  
 when $\widetilde{z}_i\ne \widetilde{z}_j$, $i\ne j$ 
 and the sum is analytically extendable to a  
rational function 
in $(\widetilde{z}_1,  \ldots, z_{\widetilde{n} } )$ 
with poles of orders less than or equal to 
$\beta(\widetilde{v}_{l', i}, \widetilde{v}_{l'', j})$ allowed  
$\widetilde{z}_i=\widetilde{z}_j$.  

For $q=2$,  
 for $\widetilde{v}_1$, $\widetilde{v}_2$, $\widetilde{v}_3 \in V$, 
 the characteristics 
$\sum_{s \in \C} \Omega(
  \mathcal Y_W (\widetilde{v}_1, \widetilde{z}_1)  \mathcal Y_W( P_s( X(\widetilde{v}_2, \widetilde{z}_2 -\zeta; 
\widetilde{v}_3, \widetilde{z}_3-\zeta)), \zeta))$      
 $+ \Omega(\widetilde{v}_1, \widetilde{z}_1;  
P_s(   \mathcal Y_V (\widetilde{v}_2, \widetilde{z}_2-\zeta) 
  \mathcal Y_V(\widetilde{v}_3; \widetilde{z}_3-\zeta) \one_V), \zeta)$,     
and 
 $\sum_{s\in \C}$  
 $\Omega ($ $P_s( Y_V(\widetilde{v}_1, \widetilde{z}_1 -\zeta)  
\mathcal Y_V(\widetilde{v}_2; \widetilde{z}_2-\zeta) \one_V), \zeta)  
 \mathcal Y_V(\widetilde{v}_3, \widetilde{z}_3)$  
$+ \mathcal Y_W(\widetilde{v}_3, \widetilde{z}_3) 
\mathcal Y_W( P_s( X(\widetilde{v}_1, \widetilde{z}_1-\zeta; \widetilde{v}_2, \widetilde{z}_2-\zeta) ), \zeta)$, 
defined on 
the domains 
$|\widetilde{z}_1-\zeta|>|\widetilde{z}_2-\zeta|$, $|\widetilde{z}_2-\zeta|>0$, 
and 
$|\zeta-\widetilde{z}_3|>|\widetilde{z}_1-\zeta|, |\widetilde{z}_2-\zeta|>0$, 
 respectively,  
are absolutely convergent and  
  analytically extendable to  
rational functions in $\widetilde{z}_1$ and $\widetilde{z}_2$ 
with poles allowed only at   
$\widetilde{z}_1$, $\widetilde{z}_2=0$, and $\widetilde{z}_1=\widetilde{z}_2$. 
\subsection{The bundle dual to $\W_{M/\F}$ on $M/\F$}   
The conditions on characteristics described in the previous Subsection  
allow to define 
 a fiber bundle on the transversal sections of $M/\F$ in the dual formulation. 
This gives us an idea how to use the notion of a dual vertex operator algebra bundle associated to 
transversal sections of a foliation. 
The condition \eqref{pairing1} for the grading together with conditions on orders  
of poles, and then the canonical pairing \eqref{pairing2}  
give rise to a pairing
  $\gamma \left(\W^\dagger_{\U}|_{ ( D^\times_{ i, p_1}, \ldots,  D^\times_{ i, p_l})    } \right)  \times 
  \gamma \left(\W_{M/\F}|_{( D^\times_{ i, p_1}, \ldots,  D^\times_{ i, p_l}) } \right)   \rightarrow \C^l$,    
for corresponding space of fibers.   
For each fiber $\mu$ of $\W_{M/\F}|_{ ( D^\times_{ i, p_1}, \ldots,  D^\times_{ i, p_l}) }$   
we obtain a linear 
operator on $\W_{\U}^\dagger|_{ ( D^\times_{ i, p_1}, \ldots,  D^\times_{ i, p_l}) }$ given by 
 this pairing.   
Thus, we obtain a well-defined linear map 
\begin{equation}
\label{tosca}
 \W_{\U}^\dagger{}_{ ( D^\times_{ i, p_1}, \ldots,  D^\times_{ i, p_l})  }: 
 \gamma\left(\W_{M/\F}|_{( D^\times_{ i, p_1}, \ldots,  D^\times_{ i, p_l}) }
\right)    
\rightarrow {\rm End} 
\left( \gamma \left(\W_{M/\F}|_{ ( D^\times_{ i, p_1}, \ldots,  D^\times_{ i, p_l})   }\right) \right),  
\end{equation}
i.e., the fibers expressed in terms of vertex operators defined on transversal sections. 
For formal coordinates $( t_{i, p_j} )$, $1 \le i \le n-p$, $1 \le j\le l$,   
on $( D^\times_{ i, p_1}, \ldots,  D^\times_{ i, p_l})$, 
 a fiber $X(v_{i, j}, z_{i, j})$     
of $\W_{M/\F}|_{  ( D^\times_{ i, p_1}, \ldots,  D^\times_{ i, p_l})   }
$   

with elements of $\overline{W}_{(p_1, \ldots, p_l)}$   
 with respect to the $(  t_{i, p_j}  )$-trivialization, the map 
 \eqref{tosca} is given by $(v_{i, j}, z_{i, j})$.   
Starting from an admissible vertex operator algebra $V$-module $W$, and applying   
 Lemma\ref{mainpro} 
we construct explicitly 
 a fiber bundle $\W_{M/\F}$ over $M/\F$,  
with canonical sections $X(p_1, \ldots, p_n)$ of 
$\W_{M/\F}|_{( D^\times_{ i, p_1}, \ldots,  D^\times_{ i, p_l})   }$ 
and fibers with values in ${\rm End} \left(\overline{W}_{(p_1, \ldots, p_n)}\right)$  
for any set of non-intersecting discs  
$( D^\times_{ i, p_1}, \ldots,  D^\times_{ i, p_l})$, $1 \le i \le n-p$ on $M/\F$.  
\subsection{A relation for $\W_M$ and $Hol(M, \F)$ cohomologies}
In this Subsection we provide Lemma \ref{groupo} 
relating the vertex operator algebra $V$ cohomology 
of the holonomy 
groupoid $Hol$( $M$, $\F)$ and $V$-bundle $\W_{M/\F}$.  
Note that the spaces $\widehat{W}^q_r$ containing non-commutative elements ${\bf X}$,   
as well as their cohomology  
are described here in terms of their characteristics.  
In \cite{Huang} it was proven that the operator $\delta^q_r X= \Delta^q_r \Omega(X)$ 
has the chain-cochain property.  
For the holonomy groupoid $Hol(M, \F)$ we obtain the linear maps 
$\delta^n_r: C_r^n \to C_{r-1}^{n+1}$, 
for each pair $l$, $r \ge 0$, and 
$\delta^2_{\vartheta}: C_{\vartheta}^2 \to C_0^3$. 
Since $C_\infty^q \subset C_r^q$
for any $r \ge 0$, and 
$C_{r_2}^q\subset C_{r_1}^q$, 
for $r_1$, $r_2\in \ge 0$ with $r_1 \le r_2$,  
$\delta_r^q.C_\infty^q$ is independent of $r$. 
 Let
$\delta_\infty^q=\delta_r^q.C_\infty^q:   
C_\infty^q 
\to C_\infty^{l+1}$.  
Thus, we obtain a double complex $(C^q_r, \delta^q_r)$, 
$(C^2_{\vartheta}), \delta^2_\vartheta)$ 
 $q$, $r \ge 0$,    
in particular, with $r=\infty$,     
$0\longrightarrow 
C_r^0
\stackrel{ \delta^0_r }{\longrightarrow} 
C_{r-1}^1  
\stackrel{\delta^1_{r-1} }{\longrightarrow} 
(C_{r-2}^2, \delta_{r, 3} \; C^2_\vartheta )
\stackrel{(\delta^2_{r-2}, \; \delta_{r, 3} \; \delta^2_\vartheta )} {\longrightarrow}
\cdots 
\stackrel{\delta^{r-1}_1}{\longrightarrow}  
C_0^r
 \longrightarrow 0$,  
with $\delta_{r-1}^{q+1} \circ  \delta^q_r=0$,  
$\delta^2_\vartheta\circ \delta^1_2=0$.  
Using the above chain complex 
one is able to introduce a cohomology of $M/\F$.   
 For $q$, $r \ge 0$ we define the 
 $q$, $r$-th vertex operator algebra cohomology $H^q_r=H^q_r(V, W, \F)$ of $M/\F$ with coefficients in 
$\overline{W}$ and depending on  
$mp$ vertex operators introduced on $\U$  
to be 
$H_r^q=
{\rm Ker} 
\left\{
\begin{array}{c}
\delta^q_r
\\
\delta_{q, 1}\; \delta^2_{\vartheta} 
\end{array}
/\mbox{\rm Im}\; \delta^{q-1}_{r+1}\right.$, including $q$,  $r=\infty$.    
Note that in general $q \le l$, $m \le p$.  
Thus the cohomology $H^q_r$ describes subdomains of lower dimensions
 on leaves of $M/\F$.  
Taking into account the content of this and previous Sections  
we obtain the following   
\begin{lemma}
\label{groupo}
The vertex operator algebra cohomology of the holonomy groupoid $Hol(M, \F)$
 is equivalent to  
 the cohomology of section spaces for  
$V$-twisted vertex algebra bundles $\W_{M/\F}$.     
\end{lemma} 
\section*{Acknowledgment}
The author is supported by the Institute of Mathematics, 
 Academy of Sciences of the Czech Republic (RVO 67985840). 

\end{document}